\documentclass[11pt]{article}%
\usepackage{amsmath}
\usepackage{amsfonts}
\usepackage{amssymb}
\usepackage{graphicx}
\usepackage{subfigure}
\providecommand{\U}[1]{\protect\rule{.1in}{.1in}}
\newtheorem{theorem}{Theorem}

\newtheorem{proposition}[theorem]{Proposition}

\newenvironment{proof}[1][Proof]{\noindent\textbf{#1.} }{\ \rule{0.5em}{0.5em}}
\begin{document}

\title{Exterior/interior problem for the circular means transform with applications
to intravascular imaging}
\author{G. Ambartsoumian
\and L. Kunyansky}
\maketitle

\begin{abstract}
Exterior inverse problem for the circular means transform (CMT) arises in the
intravascular photoacoustic imaging (IVPA), in the intravascular ultrasound
imaging (IVUS), as well as in radar and sonar. The reduction of the IPVA to
the CMT is quite straightforward. As shown in the paper, in IVUS the circular
means can be recovered from measurements by solving a certain Volterra
integral equation. Thus, a tomography reconstruction in both modalities
requires solving the exterior problem for the CMT.

Numerical solution of this problem usually is not attempted due to the
presence of \textquotedblleft invisible" wavefronts, which results in  severe
instability of the reconstruction. The novel inversion algorithm proposed in
this paper yields a stable partial reconstruction: it reproduces the
\textquotedblleft visible" part of the image and blurs the \textquotedblleft
invisible" part. If the image contains little or no invisible wavefronts (as
frequently happens in the IVPA and IVUS) the reconstruction is quantitatively
accurate. The presented numerical simulations demonstrate the feasibility of
tomography-like reconstruction in these modalities.

\end{abstract}

\section*{Introduction}

Mathematical models of various imaging modalities are based on the circular
means transform (CMT), which maps a compactly supported function of two
variables to its integrals along a two-dimensional family of circles. Some
examples of such imaging techniques include thermo- and photo-acoustic
tomography, ultrasound reflection tomography, sonar and radar imaging. Under
certain reasonable assumptions, the problem of image reconstruction in all
these procedures can be reduced to the inversion of the CMT, i.e. the
reconstruction of the unknown image function from its integrals along circles
(e.g. see \cite{Haltm-circ, Kun-cyl, Louis:2000, Cheney,Norton1, XWAK-2}, as
well as Section~\ref{S:CMT} here). The choice of the family of integration
circles (and, in particular, their location with respect to the support of the
image function) depends on the setup of the imaging device.

Typically, the centers of integration circles correspond to the locations of
transducers recording the physical signals (acoustic or electromagnetic waves)
coming from the imaging object. As a result, these centers are limited to some
discrete locations distributed along a curve, which we call a data acquisition
curve. There are abundant results about the inversion of CMT for the so-called
interior problem, arising when the data acquisition curve surrounds the
support of the image function (e.g. see \cite{Kuch-Kun} and the references
there). However only limited results are available for the exterior problem,
where the support of the image function lies outside of that curve, or the
interior/exterior problem with the support partially inside and partially
outside of that curve (e.g. see \cite{AGL} and the references there).

Some of the applications for the interior/exterior problem are radar and sonar
imaging (\cite{Louis:2000,Cheney}). However, our present work is mostly
motivated by the recent advances in intravascular ultrasound imaging (IVUS)
and intravascular photoacoustic imaging (IVPA)
\cite{Emel,Sethur-2006,Sethur-2007,Sethur-2008,BoWang}. The measuring device
in these modalities is a thin catheter with a set of embedded ultrasound
transducers. The catheter is inserted in a blood vessel and records ultrasound
waves coming from the surrounding tissue. In IVUS, the transducer first
generates an outgoing ultrasound impulse and then switches to the listening
mode and records the wave reflected by the tissue. In IVPA the acoustic waves
are created by the photoacoustic effect: a short pulse of an infrared laser
heats up the tissue that, in turn, expands and generates an outgoing acoustic wave.

To the best of our knowledge, in IVUS and IVPA a tomographic reconstruction
currently is not attempted. Instead, a crude image is reconstructed by
superimposing the incoming signal with the directional information. One of the
goals of this paper is to show that, by utilizing the equipment and measuring
techniques similar to those currently used in IVUS and IVPA, one can
reconstruct qualitatively (and sometimes even quantitatively) correct images
of certain physical properties of biological tissue. As discussed below, in
IVPA the image reconstruction problem reduces to the inversion of the CMT by
considerations well known in the standard photoacoustic tomography. The
situation in IVUS is more complicated: since the reflected wave depends on the
speed of sound one tries to recover, the inverse problem is non-linear (unlike
that of IVPA). Since in soft tissues variations of the speed of sound are not
very large, one can try to linearize the problem using the Born approximation.
In Section~\ref{S:Born} we use this approach and show that in this case the
circular integrals of the speed of sound are related to the measurements by a
Volterra integral equation (in time) whose kernel is given by a certain
integral expression. We analyze properties of the kernel and show that this
integral equation can be stably solved (by means of successive iterations or
by linear algebra techniques) to recover the circular means. This reduces the
reconstruction problem, again, to the inversion of the CMT.

The exterior inverse problem for the CMT has received so far very little
attention in the literature. The main culprit here is the well-known
ill-posedness of this problem. In order for a material interface to be
\textquotedblleft visible" under the CMT, the normal to the interface should
intersect the surface supporting the transducers (in our case, the catheter).
It is clear that in a generic exterior problem not all interfaces are visible,
and therefore accurate and stable reconstruction of material properties is not
possible. However, in the intravascular imaging the walls of blood of vessels
mostly run in the visible directions, roughly parallel to the surface of the
catheter (see, for example, \cite{Sethur-2007} for good images of
cross-sections of aorta). As our numerical simulations show (see
Section~\ref{S:numerics}), when invisible interfaces are absent, a properly
regularized algorithm can stably reconstruct a quantitatively correct image.
If the invisible interfaces are present, they will appear blurred and the
corresponding part of the image will not be reconstructed correctly. These
numerical experiments suggest the feasibility of tomography-like
reconstruction in IVPA and IVUS, which can be obtained by the techniques
proposed here.

The rest of the paper is organized as follows: In Section~\ref{S:CMT} we
discuss an exact inversion formula for the CMT in the exterior and
interior/exterior problem. We then present an algorithm based on that formula
for the approximate reconstruction of a function from its circular means. In
Section~\ref{S:Born} we show that, under Born approximation, the inverse
problem for IVUS can be reduced to the inversion of the CMT by solving a
Volterra integral equation with a continuous kernel. Section~\ref{S:numerics}
presents results of numerical simulations demonstrating the work of the
reconstructions techniques proposed in the previous sections.

\section{Exterior/interior problem for the circular means
transform\label{S:CMT}}

\subsection{Formulation of the inverse problem in IVPA\label{S:formulation}}

It is not difficult to reduce the inverse problem arising in IVPA to the
inversion of the CMT. In fact, a very similar problem arises in the
photoacoustic tomography with integrating line detectors, and the derivation
presented below can be found in the corresponding literature (see, for example
\cite{Burgh-2007}).

Let us assume that the acoustic wave is measured by infinitely long
transducers placed on the surface of a cylindrical catheter of radius $R_{0}$
and parallel to the (vertical)\ axis of the catheter. We will assume for
simplicity that the speed of sound within the catheter coincides with the
constant speed of sound in the surrounding tissue. (The constant speed
approximation is frequently made in the problems of photo- and thermoacoustic
tomography in soft tissue.) Without loss of generality this speed of sound
then can be assumed to\ be 1. The electromagnetic energy of the incoming laser
pulse is absorbed by the tissue, which leads to \ thermoelastic  expansion and
generation of an acoustic wave. An excess pressure $p(t,\mathbf{x})$ solves
the initial value problem for the wave equation in $\mathbb{R}^{3}$:%
\begin{align*}
\Delta p &  =\frac{\partial^{2}p}{\partial t^{2}},\quad\mathbf{x=(}x_{1}%
,x_{2},x_{3}\mathbf{)}\in\mathbb{R}^{3},\quad t\in(0,\infty),\\
p(0,\mathbf{x}) &  =F(\mathbf{x}),\quad\frac{\partial p(0,\mathbf{x}%
)}{\partial t}=0,
\end{align*}
where initial pressure $F(\mathbf{x})$ is determined by the properties of the
tissue; this function carries important biological information and is the
object of our interest. It is well known that the integrals of
$p(t,\mathbf{x)}$ along any selected direction solve the 2D wave equation. We,
in particular, will consider the integrals $u(t,x)$ in the direction parallel
to transducers%
\[
u(t,x)=\int\limits_{\mathbb{R}}p(t,\mathbf{x})dx_{3}.
\]
These integrals satisfy the following 2D initial value problem (IVP)
\begin{equation}
\left\{
\begin{array}
[c]{c}%
\Delta u=\frac{\partial^{2}u}{\partial t^{2}},\quad x\mathbf{=(}x_{1}%
,x_{2}\mathbf{)}\in\mathbb{R}^{2},\quad t\in(0,\infty),\\
u(0,x)=f(x),\quad\frac{\partial u(0,x)}{\partial t}=0,
\end{array}
\right.  \label{E:2dIVP}%
\end{equation}
where%
\[
f(x)=\int\limits_{\mathbb{R}}F(\mathbf{x})dx_{3}.
\]
A transducer passing through the point $z=(z_{1},z_{2})\in\mathbb{R}^{2}$ (or,
correspondingly, through the point $(z_{1,}z_{2},0)$ in $\mathbb{R}^{3})$
measures the time series $u(t,z)$ for $t\in(0,\infty).$ Since the transducers
cover the surface of the catheter, such measurements are done for all $z$
lying on the circle of radius $R_{0}$ centered at the origin in $\mathbb{R}%
^{2}.$ Ideally, we would like to reconstruct $F(\mathbf{x})$ from the
measurements $u(t,z).$ However, since the transducers are infinitely long and
integrate the data along straight lines, there is not enough data to
reconstruct $F(\mathbf{x}).$ Instead, our goal is to reconstruct $f(x).$

Solution of the IVP (\ref{E:2dIVP}) can be represented with the help of the
Green's function $G_{2D}(t,x)$ of the two-dimensional wave equation in
$\mathbb{R}^{2}$ (describing the outgoing wave)%
\begin{align}
G_{2D}(t,x)  &  =\Phi\left(  t,\sqrt{x_{1}^{2}+x_{2}^{2}}\right)  ,\qquad
x=(x_{1},x_{2})\in\mathbb{R}^{2}\backslash\Omega,\qquad t\in\mathbb{R}%
,\label{E:green-simple1}\\
\Phi\left(  t,s\right)   &  =\frac{H\left(  t,s\right)  }{2\pi\sqrt
{t^{2}-s^{2}}}, \label{E:phi-def1}%
\end{align}
where $H(t,s)$ is the Heaviside function equal to 1 for $t>s$ and equal to 0
otherwise. The well-known Kirchhoff formula \cite{Vladimirov}\ yields%
\[
u(t,z)=\frac{\partial}{\partial t}\int\limits_{\mathbb{R}^{2}}f(x)G_{2D}%
(t,x-z)dx=\frac{\partial}{\partial t}\int\limits_{\mathbb{R}^{2}}%
f(x+z)G_{2D}(t,x)dx=\frac{\partial}{\partial t}\int\limits_{0}^{\infty}%
\Phi\left(  t,r\right)  \left[  r\int\limits_{\mathbb{S}^{1}}f(z+ry)dy\right]
dr,
\]
where the term in the brackets represents the circular integrals $g(z,r)$ of
$f(x)$:%
\[
g(z,r)\equiv r\int\limits_{\mathbb{S}^{1}}f(z+ry)dy.
\]
Taking into account equations (\ref{E:green-simple1}) and (\ref{E:phi-def1})
one obtains%
\begin{equation}
u(t,z)=\frac{1}{2\pi}\frac{\partial}{\partial t}\int\limits_{0}^{t}%
\frac{g(z,r)}{2\pi\sqrt{t^{2}-r^{2}}}dr \label{E:abel}%
\end{equation}
if $t>r,$ otherwise $u(t,z)=0.$ Equation (\ref{E:abel}) is one of the versions
of the well-known, explicitly invertible Abel transform (see \cite{Burgh-2007}
or \cite{FinchRakesh} Section 4.3); for every transducer (for every $z$ with
$|z|=R_{0})$ function $g(z,r)$ can be reconstructed by the following formula:
\[
g(z,r)=4r\int\limits_{0}^{r}\frac{u(t,z)}{2\pi\sqrt{r^{2}-t^{2}}}dt.
\]
This solves the problem of finding circular integrals $g(z,r)$ (or,
equivalently, circular means $\frac{g(z,r)}{2\pi})$ from the measurements
$u(t,z)$.

Next, we develop an algorithm for the approximate reconstruction of a function
from its circular means (or circular integrals) in 2D. The centers of the
integration circles lie on the circle $S_{0}$ of radius $R_{0}$ centered at
the origin $O$, as shown in Figure 1. \begin{figure}[h]
\begin{center}
\includegraphics[width=2.7in,height=1.9in]{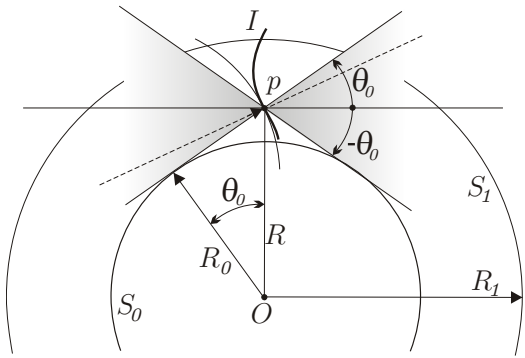}
\end{center}
\par
|\caption{Acquisition geometry and \textquotedblleft
invisible\textquotedblright\ directions}%
\label{F:geom}%
\end{figure}The function $f(x)$ is supported within a larger concentric circle
of radius $R_{1}>R_{0}.$ It will be (approximately) reconstructed from the
circular integrals $g(z,r)$ with centers $z=z(R_{0},\varphi)$ of the
integration circles sampling the whole circle $S_{0},$ i.e. $\varphi\in
\lbrack0,2\pi];$ the radii $r$ of the integration circles cover the interval
$[0,R_{0}+R_{1}].$ In other words, we consider an exterior/interior problem
for the circular means Radon transform.

The exterior problem for this transform is known to be ill-posed. In
particular, some wavefronts in the exterior of the circle $S_{0}$ are not
\textquotedblleft visible\textquotedblright, i.e. there are no integration
circles passing through the corresponding points orthogonally to the wavefront
directions. The \textquotedblleft invisible\textquotedblright\ wavefronts
cannot be stably reconstructed (e.g. see \cite{XWAK-2} and the references
there). An example of such a wavefront passing through a point $p$ is shown in
Figure 1. Here, the shaded area shows the directions of invisible
singularities at point $P$. In particular, if the function has a jump
discontinuity across interface $I$ shown in the figure, then it can not be
stably recovered from the circular means.

Since an arbitrary function has all possible wavefronts, including the
invisible ones, the general exterior (or interior/exterior) problem can only
be solved approximately. The algorithm we present below stably reconstructs
theoretically visible wavefronts and blurs the invisible ones.

\subsection{Image reconstruction from the circular means}

We start by computing the Hankel transform $\hat{g}(z,\lambda)$ of $g(z,r)/r$:%
\begin{align}
\hat{g}(z,\lambda)  &  =\int\limits_{0}^{\infty}J_{0}(\lambda
r)g(z,r)dr\label{E:ghat}\\
&  =\int\limits_{\mathbb{R}^{2}}f(z+y)J_{0}(\lambda|y|)dy=\int%
\limits_{\mathbb{R}^{2}}f(x)J_{0}(\lambda|x-z|)dx. \label{E:convolution}%
\end{align}
The next step is to substitute into (\ref{E:convolution}) the following
expression representing the well-known addition theorem (Thm. 2.10
in~\cite{Colton}) for $J_{0}:$%
\begin{align}
J_{0}(\lambda|z(R_{0},\varphi)-x(r,\theta)|)  &  =\sum_{n=-\infty}^{\infty
}J_{|n|}(\lambda R_{0})e^{in\varphi}J_{|n|}(\lambda r)e^{-in\theta
},\label{E:addition}\\
x  &  =(r\cos\theta,r\sin\theta),\quad z=R_{0}(\cos\varphi,\sin\varphi
).\nonumber
\end{align}
Since the series in (\ref{E:addition}) converges absolutely and uniformly we
obtain%
\begin{align}
\hat{g}(z(R_{0},\varphi),\lambda)  &  =\int\limits_{0}^{2\pi}\int%
\limits_{0}^{R_{0}+R_{1}}f(x(r,\theta))\sum_{n=-\infty}^{\infty}%
J_{|n|}(\lambda R_{0})e^{in\varphi}J_{|n|}(\lambda r)e^{-in\theta}%
rdrd\theta\nonumber\\
&  =\sum_{n=-\infty}^{\infty}J_{|n|}(\lambda R_{0})e^{in\varphi}%
\int\limits_{0}^{R_{0}+R_{1}}\left[  \int\limits_{0}^{2\pi}f(x(r,\theta
))e^{-in\theta}d\theta\right]  J_{|n|}(\lambda r)rdr. \label{E:circharm}%
\end{align}
Let us expand both $\hat{g}(z(R_{0},\varphi),\lambda)$ and $f(x(r,\theta))$ in
the Fourier series with respect to the angular variables:%
\begin{align}
f(x(r,\theta))  &  =\sum_{k=-\infty}^{\infty}f_{k}(r)e^{ik\theta},\qquad
f_{k}(r)=\frac{1}{2\pi}\int\limits_{0}^{2\pi}e^{-ik\theta}f(x(r,\theta
))d\theta,\label{E:ffourier}\\
\hat{g}(z(R_{0},\varphi),\lambda)  &  =\sum_{l=-\infty}^{\infty}\hat{g}%
_{l}(\lambda)e^{il\varphi},\qquad\hat{g}_{l}(\lambda)=\frac{1}{2\pi}%
\int\limits_{0}^{2\pi}e^{-il\phi}\hat{g}(z(R_{0},\varphi),\lambda)d\varphi.
\label{E:gfourier}%
\end{align}
Then, by substituting (\ref{E:circharm}) into the second equation in
(\ref{E:gfourier}) one obtains%
\begin{equation}
\hat{g}_{l}(\lambda)=2\pi J_{|l|}(\lambda R_{0})\int\limits_{0}^{\infty}%
rf_{l}(r)J_{|l|}(\lambda r)dr,\quad l\in\mathbb{Z}. \label{E:hankel}%
\end{equation}
Let us formally divide equation (\ref{E:hankel}) by $J_{|l|}(\lambda R_{0})$
(a discussion of this step follows below). We obtain the following equation%
\begin{equation}
F_{l}(\lambda)=\int\limits_{0}^{\infty}rf_{l}(r)J_{|l|}(\lambda r)dr
\label{E:hankel1}%
\end{equation}
where%
\begin{equation}
F_{l}(\lambda)\equiv\frac{\hat{g}_{l}(\lambda)}{2\pi J_{|l|}(\lambda R_{0})}.
\label{E:ratio}%
\end{equation}
The right hand side of (\ref{E:hankel1}) is the self-invertible Hankel
transform; thus $f_{l}(r)$ can be computed as%
\begin{equation}
f_{l}(r)=\int\limits_{0}^{\infty}F_{l}(\lambda)J_{|l|}(\lambda r)\lambda
d\lambda,\quad l\in\mathbb{Z}. \label{E:norton}%
\end{equation}
By combining the equations (\ref{E:ghat}), (\ref{E:gfourier}), (\ref{E:ratio}%
), (\ref{E:norton}) and (\ref{E:ffourier}) one formally reconstructs $f(x).$

This solution was first proposed by Norton~\cite{Norton1} for the interior
problem, i.e. for the case when the reconstructed function is supported in the
interior of $S_{0}.$ An important issue not addressed by Norton is the
treatment of the zeros of the Bessel functions $J_{|l|}(\lambda R_{0})$ in the
denominator of equation (\ref{E:ratio}). Since the Hankel transform $F_{l}$ of
$f_{l}$ (given by (\ref{E:hankel1})) is a bounded function of $\lambda,$ the
ratio (\ref{E:ratio}) remains bounded for all values of $\lambda,$ implying
that the exactly computed $\hat{g}_{l}(\lambda)$ vanishes at all zeros
$\lambda_{l,m},$ $(m=0,1,2,...)$ of the Bessel function $J_{l}(\lambda
R_{0}).$ Therefore, all the singularities in (\ref{E:ratio}) are removable,
for example, by application of the L'Hospital's rule. This is not a very
practical algorithmic solution since an approximation to $\hat{g}_{l}%
(\lambda)$ computed from the noisy data will not, in general, have zeros at
$\lambda_{l,m}.$ Even if one continues to formally use the L'Hospital's rule
at $\lambda_{l,m},$ computing the fraction $\hat{g}_{l}(\lambda)/J_{|l|}%
(\lambda R_{0})$ at values of $\lambda$ close to $\lambda_{l,m}$ is still an
ill-posed problem. A better solution was proposed in~\cite{Haltm-circ} where
the Fourier-Bessel series was used in a way that avoids divisions by zeros. A
more elegant solution~\cite{AmbKuch,Kun-cyl} (applicable for the interior
problem only) is to replace the Bessel function $J_{0}(\lambda r)$ in
(\ref{E:ghat}) and in (\ref{E:addition}) by the Hankel function $H_{0}%
^{(1)}(\lambda r).$ This leads to a ratio similar to (\ref{E:ratio}), but with
Bessel functions in the denominator replaced by Hankel functions. Since the
latter functions do not have zeroes for real values of $\lambda,$ a
straightforward discretization of (\ref{E:norton}) leads to a stable algorithm.

In the present case of the exterior or exterior/interior problem the methods
of~\cite{AmbKuch,Haltm-circ,Kun-cyl} are no longer applicable. In order to
avoid divisions by small values of $J_{|l|}(\lambda R_{0})$ in (\ref{E:ratio})
we will utilize two techniques. The first of them consists in using complex
values of $\lambda$ and deforming the integration contour in (\ref{E:norton})
to avoid all the zeros of $J_{|l|}(\lambda R_{0})$ except the first one at
$\lambda=0.$ In the next section we show that such a deformation does not
change the computed value of $f_{l}.$

In order to deal with the zero at $\lambda=0$ and with the small values of
$J_{|l|}(\lambda R_{0})$ in the neighborhood of this zero we will have to
restrict the values of $\lambda$ and $l$ for which $F_{l}(\lambda)$ are
computed, and to replace by zero the missing values. After such a replacement
the algorithm will no longer be theoretically exact, but the computations will
be stable. This technique is described in Section~\ref{S:regularization} .

\subsection{Deforming the contour}

In order to avoid division by zeros in (\ref{E:ratio}) we will work with
complex values of $\lambda$ lying on the curve $C$ consisting of the segments
$[0,ia]$ and infinite line $[ia,ia+M]$ with $M$ real, going to $\infty.$ The
values of $\hat{g}(z,\lambda)$ are computed by formula (\ref{E:ghat}) for all
$\lambda\in C.$ Then $f_{l}(r)$ is computed using instead of (\ref{E:norton})
the following formula%
\begin{equation}
f_{l}(r)=\int\limits_{C}F_{l}(\lambda)J_{|l|}(\lambda r)\lambda d\lambda,\quad
l\in\mathbb{Z}. \label{E:contourint}%
\end{equation}
Since all zeros of Bessel functions $J_{|l|}(t)$ are real the denominator in
the formula (\ref{E:ratio}) (defining $F_{l}(\lambda))$ vanishes only at
$\lambda=0$ (except the case $l=0$ when the denominator does not vanish at
$\lambda=0$). Numerical integration in the neighborhood of $\lambda=0$
requires regularization, as described in the next section.

We need to prove that formulas (\ref{E:norton}) and (\ref{E:contourint}) are
indeed equivalent. First, we observe the following well known property:

\begin{proposition}
For each $l\in\mathbb{Z}$, function $F_{l}(\lambda)$ given by equation
(\ref{E:hankel1}) is an entire function.
\end{proposition}

\begin{proof}
This fact follows from the well-known relation between the Hankel transforms
(\ref{E:hankel1}) and the entire 2D Fourier transform of $f(x).$ I.e., up to a
constant factor, $F_{l}(\lambda)$ is the restriction of the entire 2D Fourier
transform $\hat{f}_{l}(\xi_{1},\xi_{2})$ of a finitely supported function
$f_{l}(r)\exp(il\theta)$ to the complex plane $\xi_{1}=-\lambda\in\mathbb{C}$,
$\xi_{2}=0.$
\end{proof}

Now, in order to prove that the integral over the positive real values of
$\lambda$ can be replaced by the integral over $C$ it is enough to show that%
\begin{equation}
\lim_{b\rightarrow\infty}\int\limits_{b}^{b+ia}F_{l}(\lambda)J_{|l|}(\lambda
r)\lambda d\lambda=0. \label{E:limit}%
\end{equation}
Let us first recall the asymptotic behavior of the Bessel function $J_{\nu
}(z)$ for large values of argument$~z$ (see \cite{AbrSteg}, formula 9.2.1)
\begin{equation}
J_{\nu}(z)=\sqrt{\frac{2}{\pi z}}\left[  \cos\left(  z-\frac{\nu\pi}{2}%
-\frac{\pi}{4}\right)  +e^{|\operatorname{Im}z|}\mathcal{O}(|z|^{-1})\right]
. \label{E:bessasy}%
\end{equation}
It follows that for $\lambda$ lying on the segment $[b,b+ia]$ and $r\in
\lbrack0,R_{1}]$%
\[
|J_{|l|}(\lambda r)|\leq\sqrt{\frac{2}{\pi b}}e^{aR_{1}}\left[  1+\mathcal{O}%
\left(  \frac{1}{b}\right)  \right]  .
\]

Now, in order to prove (\ref{E:limit}) it is enough to show that
$F_{l}(\lambda)$ decays sufficiently fast in the limit $b\rightarrow\infty.$
This condition is given by the following

\begin{proposition}
For every fixed integer $l\geq0,$ in the semi-infinite strip $D=\{\lambda
\,|\;0\leq\operatorname{Im}$ $\lambda\leq a,$ $\operatorname{Re}\lambda
\geq0\}$ function $F_{l}(\lambda)$ decays uniformly with respect to
$b\equiv\operatorname{Re}\lambda$ as $o(b^{-1/2})$ when $b\rightarrow\infty.$
\end{proposition}

\begin{proof}
In addition to the asymptotic decay for large values of the argument mentioned
above, Bessel function $J_{n}(z)$ is bounded for any integer index $n$ and any
$z\in D.$ This can be easily deduced, for example, from the well-known
representation of the Bessel function (\cite{AbrSteg}, formula 9.1.21)%
\[
J_{n}(z)=\frac{1}{2\pi i^{n}}\int\limits_{0}^{2\pi}e^{iz\cos\phi}e^{in\phi
}d\phi,
\]
leading to the estimate%
\[
|J_{n}(z)|\leq e^{|\operatorname{Im}z|}.
\]
This fact can be used to derive an estimate on $F_{l}(\lambda).$ Taking into
account finite support of $f_{l}(r),$ from the definition of $F_{l}(\lambda)$
we obtain%
\begin{equation}
|F_{l}(\lambda)|=\left\vert \int\limits_{0}^{R_{1}}rf_{l}(r)J_{|l|}(\lambda
r)\,dr\right\vert \leq\int\limits_{0}^{1/\sqrt{b}}r|f_{l}(r)|\left\vert
J_{|l|}(\lambda r)\right\vert \,dr+\left\vert \int\limits_{1/\sqrt{b}}^{R_{1}%
}rf_{l}(r)J_{|l|}(\lambda r)\,dr\right\vert .\label{E:twointeg}%
\end{equation}
Since both $f_{l}(r)$ and$\ J_{|l|}(\lambda r)$ are bounded, the first
integral in the right hand side decays as $\mathcal{O(}b^{-1})$:
\begin{equation}
\int\limits_{0}^{1/\sqrt{b}}r|f_{l}(r)||J_{|l|}(\lambda r)|dr\leq\max_{0\leq
r\leq R_{1}}|f_{l}(r)|e^{aR_{1}}\int\limits_{0}^{1/\sqrt{b}}r\,dr=\mathcal{O}%
\left(  b^{-1}\right)  \text{ as }b\rightarrow\infty.\label{E:simple-est}%
\end{equation}
The last integral in (\ref{E:twointeg}) over the interval $[1/\sqrt{b},R_{1}]$
can be estimated by using the large argument asymptotics (\ref{E:bessasy}),
since $|\lambda r|\geq\sqrt{b}$ on this interval:%
\begin{align}
\left\vert \int\limits_{1/\sqrt{b}}^{R_{1}}rf_{l}(r)J_{|l|}(\lambda
r)dr\right\vert  &  =\left\vert \sqrt{\frac{2}{\pi\lambda}}\right\vert
\left\vert \int\limits_{1/\sqrt{b}}^{R_{1}}\sqrt{r}f_{l}(r)\left[  \cos\left(
\lambda r-\frac{|l|\pi}{2}-\frac{\pi}{4}\right)  +e^{|\operatorname{Im}%
(\lambda r)|}\mathcal{O}(|\lambda r|^{-1})\right]  dr\right\vert \nonumber\\
&  \leq\left\vert \sqrt{\frac{2}{\pi b}}\right\vert \left[  \left\vert
\int\limits_{1/\sqrt{b}}^{R_{1}}\sqrt{r}f_{l}(r)\cos\left(  \lambda
r-\frac{|l|\pi}{2}-\frac{\pi}{4}\right)  dr\right\vert +e^{aR_{1}}%
\int\limits_{1/\sqrt{b}}^{R_{1}}\sqrt{r}|f_{l}(r)|\mathcal{O}\left(  \frac
{1}{\sqrt{b}}\right)  dr\right]  \nonumber\\
&  =\left\vert \sqrt{\frac{2}{\pi b}}\right\vert \left\vert \int%
\limits_{1/\sqrt{b}}^{R_{1}}\sqrt{r}f_{l}(r)\cos\left(  \lambda r-\frac
{|l|\pi}{2}-\frac{\pi}{4}\right)  dr\right\vert +\mathcal{O}\left(  \frac
{1}{b}\right)  .\label{E:first-est}%
\end{align}
Next, by adding another small term (uniformly decaying as $\mathcal{O}\left(
\frac{1}{b}\right)  ),$ the integral in (\ref{E:first-est}) can be extended to
the whole interval $[0,R_{1}]$:%
\begin{equation}
\left\vert \int\limits_{1/\sqrt{b}}^{R_{1}}rf_{l}(r)J_{|l|}(\lambda
r)dr\right\vert \leq\left\vert \sqrt{\frac{2}{\pi b}}\right\vert \left\vert
\int\limits_{0}^{R_{1}}\sqrt{r}f_{l}(r)\cos\left(  \lambda r-\frac{|l|\pi}%
{2}-\frac{\pi}{4}\right)  dr\right\vert +\mathcal{O}\left(  \frac{1}%
{b}\right)  .\label{E:second-est}%
\end{equation}
Suppose that $\lambda=b+i\alpha,$ with $\alpha\in\lbrack0,a].$ For a fixed
$\alpha$ the last integral in (\ref{E:second-est}) can be re-written as a
combination of a sine and cosine Fourier transforms of a finitely supported
integrable functions $\cosh(\alpha r)\sqrt{r}f_{l}(r)$ and $\sinh(\alpha
r)\sqrt{r}f_{l}(r)$:%
\begin{align*}
\left\vert \int\limits_{0}^{R_{1}}\sqrt{r}f_{l}(r)\cos\left[  (b+i\alpha
)r-\frac{|l|\pi}{2}-\frac{\pi}{4}\right]  dr\right\vert  &  \leq\left\vert
\int\limits_{0}^{R_{1}}\sqrt{r}f_{l}(r)\cos\left(  rb-\frac{|l|\pi}{2}%
-\frac{\pi}{4}\right)  \cosh(\alpha r)dr\right\vert \\
&  +\left\vert \int\limits_{0}^{R_{1}}\sqrt{r}f_{l}(r)\sin\left(
rb-\frac{|l|\pi}{2}-\frac{\pi}{4}\right)  \sinh(\alpha r)dr\right\vert .
\end{align*}
The well-known Riemann-Lebesgue lemma then implies that for a fixed $\alpha$%
\begin{equation}
\int\limits_{0}^{R_{1}}\sqrt{r}f_{l}(r)\cos\left(  \lambda r-\frac{|l|\pi}%
{2}-\frac{\pi}{4}\right)  dr\rightarrow0\text{ as }b\rightarrow\infty
.\label{E:third-est}%
\end{equation}
We cannot claim yet that this decay is uniform with respect to
$\operatorname{Im}\lambda$. However, (\ref{E:third-est}) holds, in particular
for $\alpha=0$ and $\alpha=a,$ which correspond to the lower and upper
boundaries $\partial D_{\mathrm{lower}}$ and $\partial D_{\mathrm{upper}}$ of
the strip $D.$ Now, by combining estimates (\ref{E:twointeg}%
)-(\ref{E:third-est}), one obtains%
\begin{equation}
|F_{l}(b+i\alpha)|=o\left(  b^{-1/2}\right)  ,\text{ as }b\rightarrow
\infty,\label{E:conver2}%
\end{equation}
for every fixed $\alpha$ $\in\lbrack0,a].$ It follows that the real and
imaginary parts of the function $F_{l}(\lambda)$ are harmonic functions within
the strip $D$, vanishing at infinity. The values of these functions on
$\partial D_{\mathrm{lower}}$ and $\partial D_{\mathrm{upper}}$ decrease as
$o\left(  b^{-1/2}\right)  $ when $b\rightarrow\infty.$ Now the uniform decay
within $D$ can be established by applying the proof presented in the Appendix
of \cite{Kun-3D}, resulting in the estimate%
\begin{equation}
|F_{l}(\lambda)|=o\left(  |\operatorname{Re}\lambda|^{-1/2}\right)  ,\text{
}0\leq\operatorname{Im}\lambda\leq a,\text{ as }\operatorname{Re}%
\lambda\rightarrow\infty.\label{E:conver3}%
\end{equation}

\end{proof}

Now vanishing of the integral in (\ref{E:limit}) is guaranteed by combining
(\ref{E:conver3}) with (\ref{E:bessasy}) and we thus proved the following

\begin{theorem}
Integration over the positive real axis in (\ref{E:norton}) is equivalent to
integration over contour $C$ in (\ref{E:contourint}).
\end{theorem}

\subsection{Regularization\label{S:regularization}}

The contour deformation described in the previous section eliminates the
division by small values in the denominator of (\ref{E:ratio}) for all values
of $\lambda,$ except $\lambda=0.$ Unlike other zeros, at $\lambda=0$ Bessel
function $J_{\nu}(\lambda)$ has a root of multiplicity $|\nu|$ (we consider
only integer $\nu)$. In addition $J_{\nu}(t)$ remains small (much less than
one) in the interval $t\in\lbrack0,\nu(1-\varepsilon))]$ (with the value of
$\varepsilon$ depending on the desired notion of \textquotedblleft
smallness\textquotedblright). A more quantitative description of this decay is
given by the first term of the Debye asymptotics (\cite{AbrSteg}, formula
9.3.7):%
\[
J_{\nu}(\nu\operatorname{sech}\alpha)\thicksim\frac{1}{\sqrt{2\pi\nu
\tanh\alpha}}\exp[\nu(\tanh\alpha-\alpha)],\quad\alpha>0,\quad\nu
\rightarrow+\infty.
\]
It follows that in any direction $x=k\nu,$ with $|k|<1,$ values of $J_{\nu
}(x)$ decay exponentially as $x\rightarrow\infty.$

Such behavior of Bessel functions results in extremely small values of
$F_{l}(\lambda)$ and $\hat{g}_{l}(\lambda)$ for certain values of $\lambda.$
Since $f_{l}(r)$ are finitely supported in the interval $[0,R_{1}]$, equation
(\ref{E:hankel1}) shows that $F_{l}(\lambda)$ become very small for values of
$|l|\geq\lambda R_{1}.$ For simplicity we neglected here the narrow transition
zone. We also notice that this effect is preserved when $\lambda$ is slightly
shifted into the imaginary direction ($|l|\gg1\geq|\operatorname{Im}%
\lambda|).$ In other words, the vanishing values lie within the cone $|l|\geq
R_{1}\operatorname{Re}\lambda.$

We notice, parenthetically, that this effect is closely related to the
well-known \textquotedblleft bow-tie" shape of the Fourier spectrum of
projections in classical X-ray tomography. Functions $F_{l}(\lambda)$ can be
understood as the Fourier coefficients obtained by expanding the Fourier
transform $\hat{f}(\xi)$ of $f(x)$ represented in polar coordinates in the 1D
Fourier series in the angular coordinate. It follows that if $f(x)$ is
supported within a circle of radius $R_{1}$ then there is little energy in the
cone $|l|\geq R_{1}\lambda$; these values do not need to be computed and can
be set to zero without much effect on the reconstructed image.

However, functions $\hat{g}_{l}(\lambda)$ given by (\ref{E:hankel}) have an
additional Bessel factor $J_{|l|}(\lambda R_{0})$ in front of the integral.
According to the asymptotic behavior of the Bessel functions, the exact values
of $\hat{g}_{l}(\lambda)$ will become very small in the larger cone $|l|\geq
R_{0}\operatorname{Re}\lambda.$ Since we need to compute (\ref{E:ratio}) from
the approximately known values of $\hat{g}_{l}(\lambda),$ the division within
the cone $|l|\geq R_{0}\operatorname{Re}\lambda$ is an ill-posed operation and
should be avoided. The regularization step of the algorithm consists in
replacing values of $F_{l}(\lambda)$ by zeros within the cone $|l|\geq
R_{0}\operatorname{Re}\lambda.$ While values of $F_{l}(\lambda)$ within the
smaller region of $|l|\geq R_{1}\operatorname{Re}\lambda$ are very small and
can be set to zero without noticeably affecting the image, the additional loss
of values in the region $R_{0}\operatorname{Re}\lambda\leq|l|\leq
R_{1}\operatorname{Re}\lambda$ will cause the disappearance of certain
material interfaces (or wavefronts).

It is interesting to investigate which wavefronts will remain in the image and
which will be smoothed out. Notice that the reconstructed image is obtained by
combining equations (\ref{E:ffourier}) and (\ref{E:contourint}), in other
words
\[
f(x(r,\theta))=\sum_{l=-\infty}^{\infty}\int\limits_{C}F_{l}(\lambda
)J_{|l|}(\lambda r)e^{il\theta}\lambda d\lambda.
\]
Setting to zero certain values of $F_{l}(\lambda)$ is equivalent to removing
from the reconstructed image components in the form $J_{|l|}(\lambda
r)e^{il\theta}$ for certain values of $\lambda$ (for simplicity we will ignore
in the following crude analysis the small imaginary part of $\lambda$).

Let us analyze local behavior of a function $J_{|l|}(\lambda r)e^{il\theta}$
in a small neighborhood $N$ of point $p$ as shown in Figure 1. (Due to the
rotational symmetry of the acquisition scheme it is enough to understand this
behavior only for the points located on the vertical axis). Within $N$ we have
$x_{1}\thickapprox R\theta,$ so that the angular exponent $e^{il\theta}$ can
be approximated by $\exp(ilx_{1}/R),$ where $x=(x_{1},x_{2})$. The function
$u(x(r,\theta))\equiv J_{|l|}(\lambda r)e^{il\theta}$ solves the Helmholtz
equation%
\[
\Delta u+\lambda^{2}u=0.
\]
Locally the dependence of this function on $x_{1}$ can be approximated by
$\exp(ilx_{1}/R).$ Then $u(x)$ can be approximated by the formula
\begin{equation}
u(x)\thickapprox\exp(ilx_{1}/R)(c_{1}\exp(i\gamma x_{2})+c_{2}\exp(-i\gamma
x_{2})) \label{E:twowaves}%
\end{equation}
with some constants $c_{1}$ and $c_{2},$ and with the vertical frequency
$\gamma$ satisfying the condition%
\[
\frac{l^{2}}{R^{2}}+\gamma^{2}=\lambda^{2}.
\]
The two plane waves given by~(\ref{E:twowaves}) propagate in the directions
given by the wave vectors $v^{+}$ and $v^{-}$:
\[
v^{\pm}=(l/R,\pm\gamma)=\lambda(\cos\theta,\sin(\pm\theta)),
\]
where the angle $\theta$ between $v^{+}$ and the horizontal axis is given by
the condition
\[
\cos\theta=\frac{l}{\lambda R}.
\]

Now, if components $J_{|l|}(\lambda r)e^{il\theta}$ with $\lambda R_{0}%
\leq|l|$ are filtered out by the regularization step of the algorithm, the
boundary of the cone of excluded directions are given by the condition%
\[
\lambda R_{0}=|l|
\]
or%
\[
\cos\theta_{0}=\frac{\lambda R_{0}}{\lambda R}=\frac{R_{0}}{R}.
\]
However, this condition coincides with the visibility condition discussed in
the beginning of Section~\ref{S:CMT}. It follows that our regularization
technique removes from the image plane waves propagating in the
\textquotedblleft invisible\textquotedblright\ directions but keeps all the others.

\section{Inverse problem for IVUS and the CMT\label{S:Born}}

While the inverse problem of IVPA is reduced to the inverse problems for the
CMT in a rather straightforward way (see Section \ref{S:formulation}), the
situation in IVUS is more complicated. The important difference between these
modalities lies in the duration of the excitation. In IVPA the acoustic wave
is excited by a short laser pulse; it's duration is much shorter than the time
needed for an acoustic wave to cover a distance compared to the desired
resolution of the image. Thus, the forward problem can be modelled by the wave
equation without sources (the source term is absorbed in the initial condition
at $t=0).$ In the IVUS, on the other hand, the reflected acoustic wave that
carries the desired information is generated by the incoming excitation wave
initiated by the transducer; both waves propagate with the same speed. In
addition, while in the IVPA the speed of sound can be assumed constant and
known, in the IVUS the variations in the speed of sound is the information one
wants to reconstruct. This makes the inverse problem non-linear; we will
linearize it by using the Born approximation. In this section we show that
under the latter approximation the circular means of the speed of sound
(centered at each transducer) can be reconstructed from the measurements by
solving a certain Volterra integral equation. After the circular means are
found, the exterior problem for the CMT can be solved by \ the methods
presented in the previous section.

We will assume that each transducer is infinitely long and the measurements
are made by each transducer sequentially and independently from the others. A
transducer initiates an outgoing excitation wave $u_{\mathrm{exc}%
}(t,\mathbf{x})$ and then switches into the recording mode. The transducers
are presumed to be infinitely thin, so that after the initial pule is
generated, they do not perturb the propagation of acoustic waves. Under these
assumptions function $u_{\mathrm{exc}}(t,\mathbf{x})$ satisfies the free-space
wave equation
\begin{equation}
\frac{\partial^{2}}{\partial t^{2}}u_{\mathrm{exc}}(t,\mathbf{x}%
)=c^{2}(\mathbf{x})\Delta u_{\mathrm{exc}}(t,\mathbf{x}),\qquad\mathbf{x}%
\equiv(x_{1},x_{2},x_{3})\in\mathbb{R},\qquad t\in(0,\infty).
\label{E:full-wave}%
\end{equation}
We will assume that the speed of sound $c(\mathbf{x})$ is close to a constant
and this constant (by choosing proper physical units) can be made equal to
unity. In other words%
\begin{equation}
c^{2}(\mathbf{x})=1+m(\mathbf{x}),\qquad|m(\mathbf{x})|\ll1. \label{E:m-x}%
\end{equation}
Now (\ref{E:full-wave}) can be re-written in the following form:
\begin{equation}
\frac{\partial^{2}}{\partial t^{2}}u_{\mathrm{exc}}(t,\mathbf{x}%
)=(1+m(\mathbf{x}))\Delta u_{\mathrm{exc}}(t,\mathbf{x}),\qquad\mathbf{x}%
\in\mathbb{R}^{3},\qquad t\in(0,\infty). \label{E:full-wave1}%
\end{equation}
Let us also consider the solution $u_{0}(t,\mathbf{x})$ of the wave equation
in a homogeneous media (excited by the same transducer). It satisfies the
homogeneous wave equation%
\begin{equation}
\frac{\partial^{2}}{\partial t^{2}}u_{0}(t,\mathbf{x})=\Delta u_{0}%
(t,\mathbf{x}),\qquad\mathbf{x}\in\mathbb{R}^{3},\qquad t\in(0,\infty).
\label{E:homo}%
\end{equation}
Now the difference $w(t,\mathbf{x})=u_{\mathrm{exc}}(t,\mathbf{x}%
)-u_{0}(t,\mathbf{x})$ satisfies the equation
\[
\frac{\partial^{2}}{\partial t^{2}}w(t,\mathbf{x})=\Delta w(t,\mathbf{x}%
)+m(\mathbf{x})\Delta u_{0}(t,\mathbf{x})+m(\mathbf{x})\Delta w(t,\mathbf{x}%
),\qquad\mathbf{x}\in\mathbb{R}^{3},\qquad t\in(0,\infty).
\]
The Born approximation results from neglecting the last term in the above
equation (this is a second order term with respect to $m(\mathbf{x})$ since
$\Delta w$ has the same order as $m(\mathbf{x})).$ Taking into account
(\ref{E:homo}) we obtain%
\begin{align}
\frac{\partial^{2}}{\partial t^{2}}w(t,\mathbf{x})  &  =\Delta w(t,\mathbf{x}%
)+m(\mathbf{x})\Delta u_{0}(t,\mathbf{x})\nonumber\\
&  =\Delta w(t,\mathbf{x})+m(\mathbf{x})\frac{\partial^{2}}{\partial t^{2}%
}u_{0}(t,\mathbf{x}),\qquad\mathbf{x}\in\mathbb{R}^{3},\qquad t\in(0,\infty).
\label{E:born}%
\end{align}

Our goal is to reconstruct from the measurements some information about
$m(\mathbf{x}).$ Since the transducers we model do not have any resolution in
the vertical direction, we will only be able to partially reconstruct the
integrals $\overline{m}(x_{1},x_{2})$ of $m(\mathbf{x})$ in $x_{3}$:
\[
\overline{m}(x_{1},x_{2})=\int\limits_{\mathbb{R}}m(\mathbf{x})dx_{3}.
\]
The first step towards this goal is to reconstruct from the measurements
obtained by one transducer (without loss of generality assumed to be located
along the $x_{3}$ axis) the circular averages $M(r)$ of $\overline{m}%
(x_{1},x_{2})$ defined as follows%
\[
M(r)=\int\limits_{0}^{2\pi}\overline{m}(r\cos\theta,r\sin\theta)d\theta.
\]

If the wave $u_{0}(t,\mathbf{x})$ is excited by an infinitely thin transducer
lying on the $Ox_{3}$ axis, $u_{0}(t,\mathbf{x})$ should be invariant with
respect to $x_{3},$ and invariant with the respect to rotations about the axis
$Ox_{3}.$ We will represent $u_{0}$ with the help of the Green's function
$G_{2D}(t,x_{1},x_{2})$ of the two-dimensional wave equation in $\mathbb{R}%
^{2}$ satisfying the radiation condition at infinity and given by equations
(\ref{E:green-simple1}) and (\ref{E:phi-def1}):
\begin{align}
u_{0}(t,\mathbf{x)}  &  \mathbf{=}u_{0}^{\ast}\left(  t,\sqrt{x_{1}^{2}%
+x_{2}^{2}}\right)  \equiv\int_{\mathbb{R}}\varphi(\tau)G_{2D}(t-\tau
,x_{1},x_{2})d\tau\label{E:to2D}\\
&  =\int_{\mathbb{R}}\varphi(\tau)\Phi\left(  t-\tau,\sqrt{x_{1}^{2}+x_{2}%
^{2}}\right)  d\tau\nonumber\\
&  =\int_{\mathbb{R}}\varphi(\tau)\Phi\left(  t-\tau,r(\mathbf{x})\right)
d\tau, \label{E:exc}%
\end{align}
where $r(\mathbf{x})=\sqrt{x_{1}^{2}+x_{2}^{2},}$ and $\varphi(\tau)$ is a
$C^{\infty}$ function in $\mathbb{R}$ finitely supported within the interval
$[-a,0].$ The function $\varphi(\tau)$ is a delta-approximating function
describing the initial pressure on the surface of the transducer. It is
introduced in order to avoid some technical difficulties; further in this
section we will pass to the limit $a\rightarrow0$ and $\varphi(\tau
)\rightarrow\delta(\tau).$ We notice that for such a choice of $\varphi
(\tau),$ the function $u_{0}(t,\mathbf{x)}$ still solves the wave equation in
the whole space $\mathbb{R}^{2}$ for all $t\geq0,$ and this solution is
invariant with respect to $x_{3}$ (it depends only on $r(\mathbf{x})$ and
$t).$

Consider the free space Green's function $G_{3D}(t,\mathbf{x})$ of the 3D wave
equation satisfying the radiation condition at infinity%
\[
G_{3D}(t,\mathbf{x})=\frac{\delta(|\mathbf{x}|-t)}{4\pi t}.
\]
Using $G_{3D}(t,\mathbf{x}),$ solution of equation (\ref{E:born}) can be
re-written in the form%
\[
w(s,\mathbf{x})=\int\limits_{\mathbb{R}^{3}}\int\limits_{\mathbb{R}}\left[
m(\mathbf{y})\frac{\partial^{2}}{\partial t^{2}}u_{0}(t,\mathbf{y})\right]
G_{3D}(s-t,\mathbf{x-y})dtd\mathbf{y.}%
\]
Under an additional assumption that $m(\mathbf{x})$ is finitely supported in
space in $x_{3}$ (or that it decreases at infinity sufficiently fast), one can
consider integrals $\bar{w}(t,x_{1},x_{2})$ of $w(t,\mathbf{x})$ in $x_{3}$:%
\[
\bar{w}(t,x_{1},x_{2})\equiv\int\limits_{\mathbb{R}}w(t,\mathbf{x})dx_{3}%
\]

The transducer measures $\bar{w}(s,0,0),$ i.e. the integral of $w(s,\mathbf{x}%
)$ over the $Ox_{3}$ axis:%
\begin{align}
\bar{w}(s,0,0)  &  =\int\limits_{\mathbb{R}}\int\limits_{\mathbb{R}^{3}}%
\int\limits_{\mathbb{R}}\left[  m(\mathbf{y})\frac{\partial^{2}}{\partial
t^{2}}u_{0}(t,\mathbf{y})\right]  G_{3D}(s-t,(0,0,x_{3})-\mathbf{y}%
)dtd\mathbf{y}dx_{3}\nonumber\\
&  =\int\limits_{\mathbb{R}}\int\limits_{\mathbb{R}}\int\limits_{\mathbb{R}%
}\int\limits_{\mathbb{R}}\left[  m(\mathbf{y})\frac{\partial^{2}}{\partial
t^{2}}u_{0}(t,\mathbf{y})\right]  G_{2D}(s-t,y_{1},y_{2})dtdy_{1}dy_{2}dy_{3}
\label{E:four-ints}%
\end{align}
where we interchanged the order of integrations and made use of the fact that%
\[
G_{2D}(t,-y_{1},-y_{2})=G_{2D}(t,y_{1},y_{2})=\int\limits_{\mathbb{R}}%
G_{3D}(t,\mathbf{(}y_{1},y_{2},y_{3}\mathbf{)})dy_{3}.
\]
By using (\ref{E:to2D}) equation (\ref{E:four-ints}) can be further simplified
as follows%
\begin{align}
\bar{w}(s,0,0)  &  =\int\limits_{\mathbb{R}}\int\limits_{\mathbb{R}}%
\int\limits_{\mathbb{R}}\left[  \int\limits_{\mathbb{R}}m(\mathbf{y}%
)\frac{\partial^{2}}{\partial t^{2}}u_{0}^{\ast}\left(  t,\sqrt{y_{1}%
^{2}+y_{2}^{2}}\right)  dy_{3}\right]  G_{2D}(s-t,y_{1},y_{2})dtdy_{1}%
dy_{2}\nonumber\\
&  =\int\limits_{\mathbb{R}}\int\limits_{\mathbb{R}}\int\limits_{\mathbb{R}%
}\left[  \overline{m}(y_{1},y_{2})\frac{\partial^{2}}{\partial t^{2}}%
u_{0}^{\ast}\left(  t,\sqrt{y_{1}^{2}+y_{2}^{2}}\right)  \right]
G_{2D}(s-t,y_{1},y_{2})dtdy_{1}dy_{2}. \label{E:three-ints}%
\end{align}
By utilizing formula (\ref{E:green-simple1}) and by integrating
(\ref{E:three-ints}) in polar coordinates, $\bar{w}(s,0,0)$ can be expressed
in terms of the circular averages $M(r)$ as follows
\begin{align}
\bar{w}(s,0,0)  &  =\int\limits_{\mathbb{R}}\int\limits_{\mathbb{R}}%
\int\limits_{\mathbb{R}}\left[  \overline{m}(y_{1},y_{2})\frac{\partial^{2}%
}{\partial t^{2}}u_{0}^{\ast}\left(  t,\sqrt{y_{1}^{2}+y_{2}^{2}}\right)
\right]  \Phi\left(  s-t,\sqrt{y_{1}^{2}+y_{2}^{2}}\right)  dtdy_{1}%
dy_{2}\nonumber\\
&  =\int\limits_{0}^{2\pi}\int\limits_{0}^{\infty}\int\limits_{\mathbb{R}%
}\left[  \overline{m}(r\cos\theta,r\sin\theta)\frac{\partial^{2}}{\partial
t^{2}}u_{0}^{\ast}\left(  t,r\right)  \right]  \Phi\left(  s-t,r\right)
dtrdrd\theta\nonumber\\
&  =\int\limits_{0}^{\infty}\int\limits_{\mathbb{R}}\left[  M(r)\frac
{\partial^{2}}{\partial t^{2}}u_{0}^{\ast}\left(  t,r\right)  \right]
\Phi\left(  s-t,r\right)  dtrdr. \label{E:two-ints}%
\end{align}
\ \ The substitution of (\ref{E:exc}) into (\ref{E:two-ints}) results in the
following integro-differential equation relating circular averages $M(r)$ with
the measurements $\bar{w}(s,0,0)$:
\begin{equation}
\bar{w}(s,0,0)=\int\limits_{0}^{\infty}\int\limits_{\mathbb{R}}M(r)\left[
\frac{\partial^{2}}{\partial t^{2}}\int_{\mathbb{R}}\varphi(\tau)\Phi\left(
t-\tau,r\right)  d\tau\right]  \Phi\left(  s-t,r\right)  dtrdr.
\label{E:complicated}%
\end{equation}
Below we show that by a proper choice of $\varphi(\tau)$ the above equation
can be reduced to the Volterra integral equation of the second kind.

Let us consider the anti-derivative $\Psi\left(  t,r\right)  $ of $\Phi\left(
t,r\right)  $ in $t$:%
\[
\frac{\partial}{\partial t}\Psi\left(  t,r\right)  =\Phi\left(  t,r\right)
,\qquad\Phi\left(  t-\tau,r\right)  =-\frac{\partial}{\partial\tau}\Psi\left(
t-\tau,r\right)  .
\]
Then, by taking into account the finite support of $\varphi(\tau),$ the
expression in the brackets in (\ref{E:complicated}) can be transformed as
follows:%
\begin{align*}
\frac{\partial^{2}}{\partial t^{2}}\int_{\mathbb{R}}\varphi(\tau)\Phi\left(
t-\tau,r\right)  d\tau &  =\frac{\partial^{2}}{\partial t^{2}}\int%
_{\mathbb{R}}\varphi^{\prime}(\tau)\Psi\left(  t-\tau,r\right)  d\tau\\
&  =\frac{\partial}{\partial t}\int_{\mathbb{R}}\varphi^{\prime}(\tau
)\Phi\left(  t-\tau,r\right)  d\tau=\int_{\mathbb{R}}\varphi^{\prime\prime
}(\tau)\Phi\left(  t-\tau,r\right)  d\tau.
\end{align*}
Now (\ref{E:complicated}) can be re-written in the following form%
\begin{align*}
\bar{w}(s,0,0)  &  =\int\limits_{0}^{\infty}\int\limits_{\mathbb{R}%
}M(r)\left[  \int\limits_{\mathbb{R}}\varphi^{\prime\prime}(\tau)\Phi\left(
t-\tau,r\right)  d\tau\right]  \Phi\left(  s-t,r\right)  dtrdr\\
&  =\int\limits_{0}^{\infty}rM(r)\left(  \int\limits_{\mathbb{R}}%
\varphi^{\prime\prime}(\tau)\left[  \int\limits_{\mathbb{R}}\Phi\left(
t-\tau,r\right)  \Phi\left(  s-t,r\right)  dt\right]  d\tau\right)  dr\\
&  =\int\limits_{0}^{\infty}rM(r)\left(  \int\limits_{\mathbb{R}}%
\varphi^{\prime\prime}(\tau)\left[  \int\limits_{\mathbb{R}}\Phi\left(
t,r\right)  \Phi\left(  (s-\tau)-t,r\right)  dt\right]  d\tau\right)  dr\\
&  =\int\limits_{0}^{\infty}rM(r)\left(  \int\limits_{\mathbb{R}}%
\varphi^{\prime\prime}(\tau)K(r,s-\tau)d\tau\right)  dr
\end{align*}
where%
\begin{equation}
K(r,v)\equiv\int\limits_{\mathbb{R}}\Phi\left(  t,r\right)  \Phi\left(
v-t,r\right)  dt. \label{E:k-def}%
\end{equation}
Let us consider the case of $\varphi(\tau)$ equal to the Dirac's delta
function $\delta(\tau).$ To this end introduce a family of delta-approximating
functions $\varphi_{\alpha}(\tau)$ and the corresponding family of
measurements $\bar{w}_{\alpha}(s,0,0)$ so that%
\[
\lim_{\alpha\rightarrow0}\varphi_{\alpha}(\tau)=\delta(\tau).
\]
Then, taking the limit $\alpha\rightarrow0$ yields
\begin{align*}
\bar{w}_{\delta}(s,0,0)  &  \equiv\lim_{\alpha\rightarrow0}\bar{w}_{\alpha
}(s,0,0)=\int\limits_{0}^{\infty}rM(r)\left(  \int\limits_{\mathbb{R}}%
\delta^{\prime\prime}(\tau)K(r,s-\tau)d\tau\right)  dr\\
&  =\int\limits_{0}^{\infty}rM(r)\frac{\partial^{2}}{\partial s^{2}}K(r,s)dr,
\end{align*}
where the second derivative of $K(r,s)$ should be understood in the sense of
distributions. In fact, it will be more convenient for us to work with the
anti-derivative of the data $W_{\delta}(s)\equiv\int\bar{w}_{\delta
}(s,0,0)ds.$ The latter function is related to the averages $M(r)$ by the
equation%
\begin{equation}
W_{\delta}(s)=\int\limits_{0}^{\infty}rM(r)\frac{\partial}{\partial
s}K(r,s)dr, \label{E:nice}%
\end{equation}
where the derivative is, again, understood in the sense of distributions.

Let us investigate the behavior of $K(r,s).$ By combining (\ref{E:phi-def1})
and (\ref{E:k-def}) we obtain%
\[
K(r,s)=\int\limits_{\mathbb{R}}\Phi\left(  t,r\right)  \Phi\left(
s-t,r\right)  dt=\frac{1}{4\pi^{2}}\int\limits_{\mathbb{R}}\frac{H\left(
t,r\right)  }{\sqrt{t^{2}-r^{2}}}\frac{H\left(  s-t,r\right)  }{\sqrt
{(s-t)^{2}-r^{2}}}dt.
\]

We observe that if $s\leq2r,$ the numerator of the integrand identically
vanishes and $K(r,s)=0.$ Otherwise, if $s>2r,$%
\[
K(r,s)=\frac{1}{4\pi^{2}}\int\limits_{r}^{s-r}\frac{1}{\sqrt{(s-t)^{2}-r^{2}%
}\sqrt{t^{2}-r^{2}}}dt.
\]
By substitution $q=t-s/2$ this can be further simplified to
\[
4\pi^{2}K(r,s)=\int\limits_{0}^{\frac{s}{2}-r}\frac{2}{\sqrt{\left[  \left(
\frac{s}{2}-q\right)  ^{2}-r^{2}\right]  \left[  \left(  \frac{s}{2}+q\right)
^{2}-r^{2}\right]  }}dq=\int\limits_{0}^{\frac{s}{2}-r}\frac{2}{\sqrt{\left[
\left(  \frac{s}{2}-r\right)  ^{2}-q^{2}\right]  \left[  \left(  \frac{s}%
{2}+r\right)  ^{2}-q^{2}\right]  }}dq.
\]
Integration by parts, followed by substitution $\eta=q/(s/2-r)$ further yields%
\begin{align}
4\pi^{2}K(r,s)  &  =\frac{\pi}{\sqrt{2rs}}-2\int\limits_{0}^{s/2-r}%
\arcsin\left(  \frac{q}{s/2-r}\right)  \frac{q}{\sqrt{(s/2+r)^{2}-q^{2}}^{3}%
}dq\nonumber\\
&  =\frac{\pi}{\sqrt{2rs}}-2(s/2-r)^{2}\int\limits_{0}^{1}\frac{\eta
\arcsin\eta}{\sqrt{(s/2+r)^{2}-(s/2-r)^{2}\eta^{2}}^{3}}d\eta,\qquad s>2r.
\label{E:kernel-int}%
\end{align}
We notice that for $r>0$ the\ expression in the denominator of the last
integrand is bounded away from zero:%
\[
(s/2+r)^{2}-\eta^{2}(s/2-r)^{2}\geq(s/2+r)^{2}-(s/2-r)^{2}=4s^{2}r^{2}%
>16r^{4},\qquad s>2r.
\]
Thus, as $s$ approaches $2r,$ the second term in (\ref{E:kernel-int}) vanishes
as $\mathcal{O}\left(  (s/2-r)^{2}\right)  $. Moreover, it is easy to check
that the derivative (in $s)$ of this term also vanishes as $s$ approaches
$2r.$ Therefore, as a function of $s,$ $K(r,s)$ has a jump at $s=2r$ which can
be represented using the Heaviside function in the form $\frac{\pi}%
{2r}H(s,2r)$ with the remaining part continuous through the point $s=2r$:%
\begin{align*}
4\pi^{2}K(r,s)  &  =\frac{\pi}{2r}H(s,2r)+K_{1}(r,s),\\
K_{1}(r,s)  &  \equiv\left\{
\begin{array}
[c]{cc}%
0, & s\leq2r\\
4\pi^{2}K(r,s)-\frac{\pi}{2r}H(s,2r), & s>2r.
\end{array}
\right.
\end{align*}
Now, the derivative of $K(r,s)$ has form%
\[
\frac{\partial}{\partial s}K(r,s)=\frac{1}{8\pi r}\delta(s-2r)+\frac{1}%
{4\pi^{2}}\frac{\partial}{\partial s}K_{1}(r,s),
\]
where the second term is a bounded (although discontinuous at $s=2r)$
function. The second term can be extended by continuity to the point $s=2r.$
Substituting this expression in (\ref{E:nice}) yields
\[
W_{\delta}(s)=\frac{1}{8\pi s}M(s/2)+\frac{1}{4\pi^{2}}\int\limits_{0}%
^{s/2}rM(r)\frac{\partial}{\partial s}K_{1}(r,s)dr.
\]
This is a Volterra integral equation of the second kind, with a continuous
kernel. Equations of this type are well-posed and have unique solutions. They
are also easy to solve numerically. For example, the method of successive
approximations will converge unconditionally; it consists in computing the
next approximation $M_{(k+1)}$ from the previous one $M_{(k)}$ by the formula%
\[
M_{(k+1)}(s/2)=8\pi sW_{\delta}(s)-\frac{2s}{\pi}\int\limits_{0}^{s/2}%
rM_{(k)}(r)\frac{\partial}{\partial s}K_{1}(r,s)dr.
\]

Alternatively, our analysis of the kernel $\frac{\partial}{\partial s}K(r,s)$
suggests that, when properly discretized, equation (\ref{E:nice}) will reduce
to a well-posed system of linear equations, with a triangular matrix. We chose
this alternative to conduct numerical simulations described in the next
section. \begin{figure}[t]
\begin{center}
\subfigure[]{\includegraphics[width=1.7in,height=1.7in]{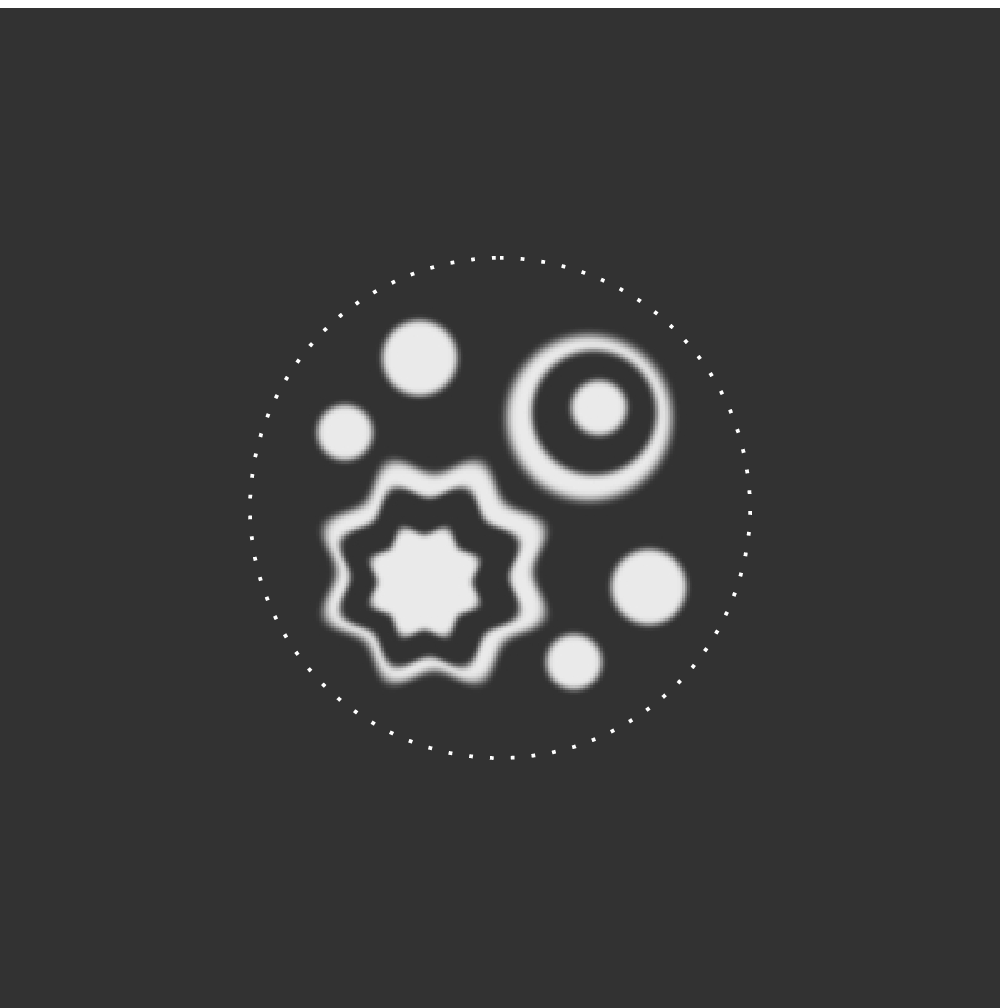}}
\subfigure[]{\includegraphics[width=1.7in,height=1.7in]{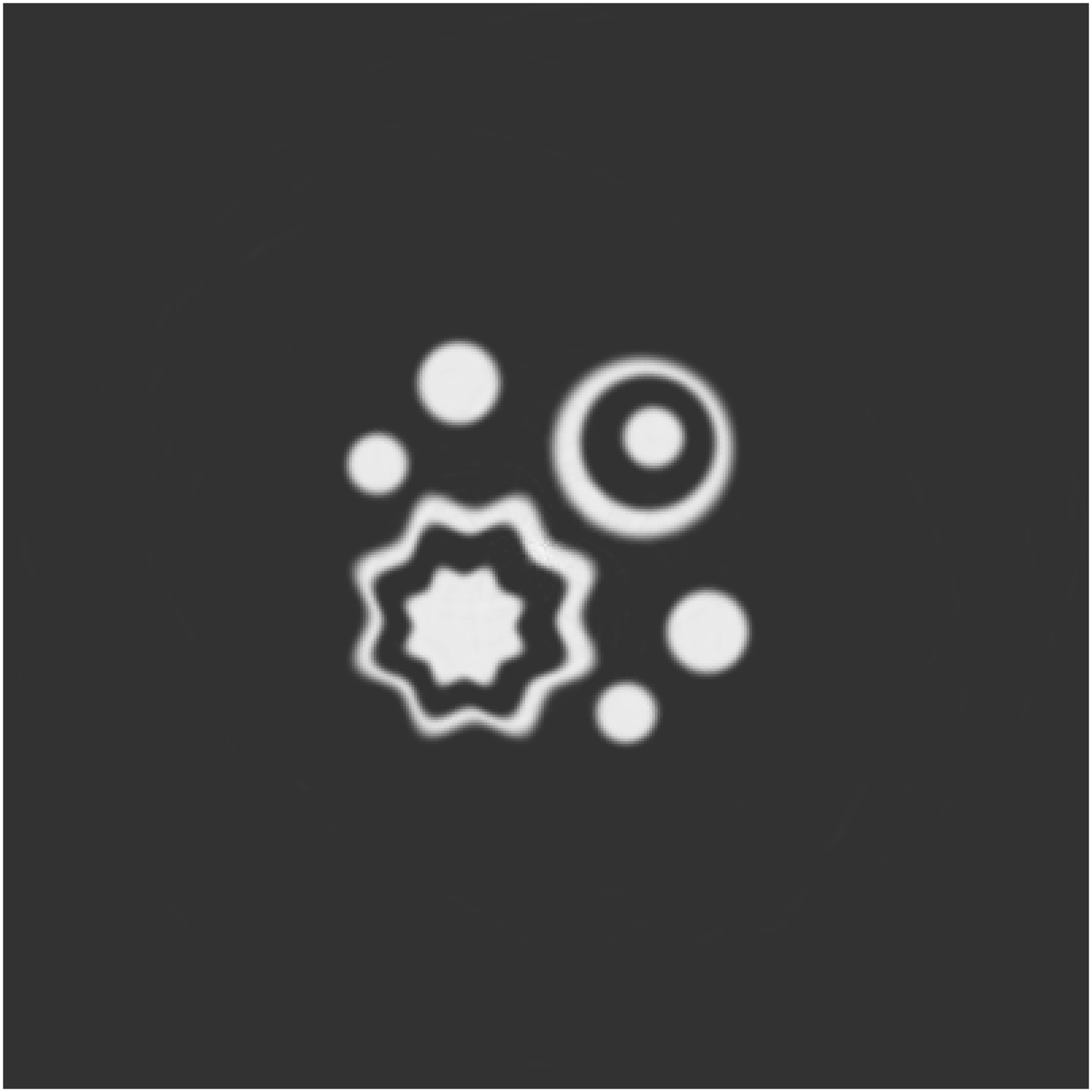}}\\
\subfigure[]{\includegraphics[width=1.7in,height=1.7in]{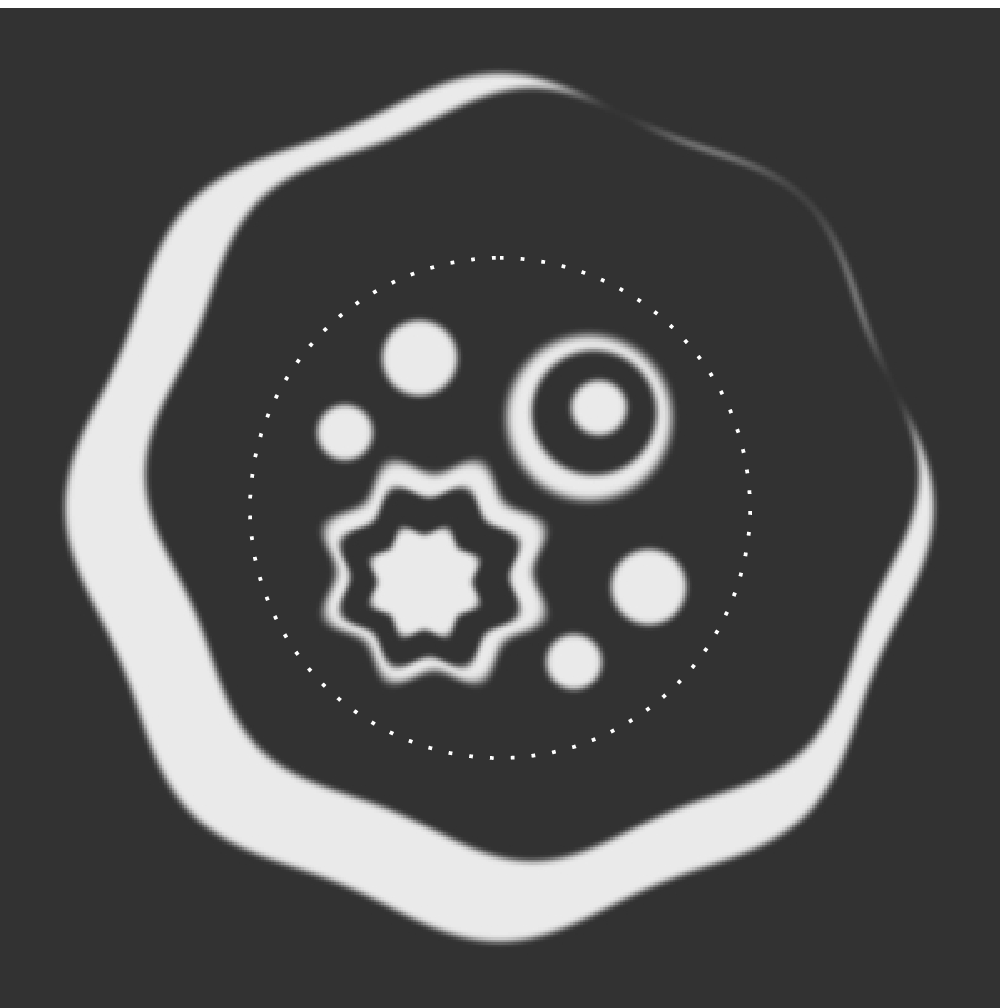}}
\subfigure[]{\includegraphics[width=1.7in,height=1.7in]{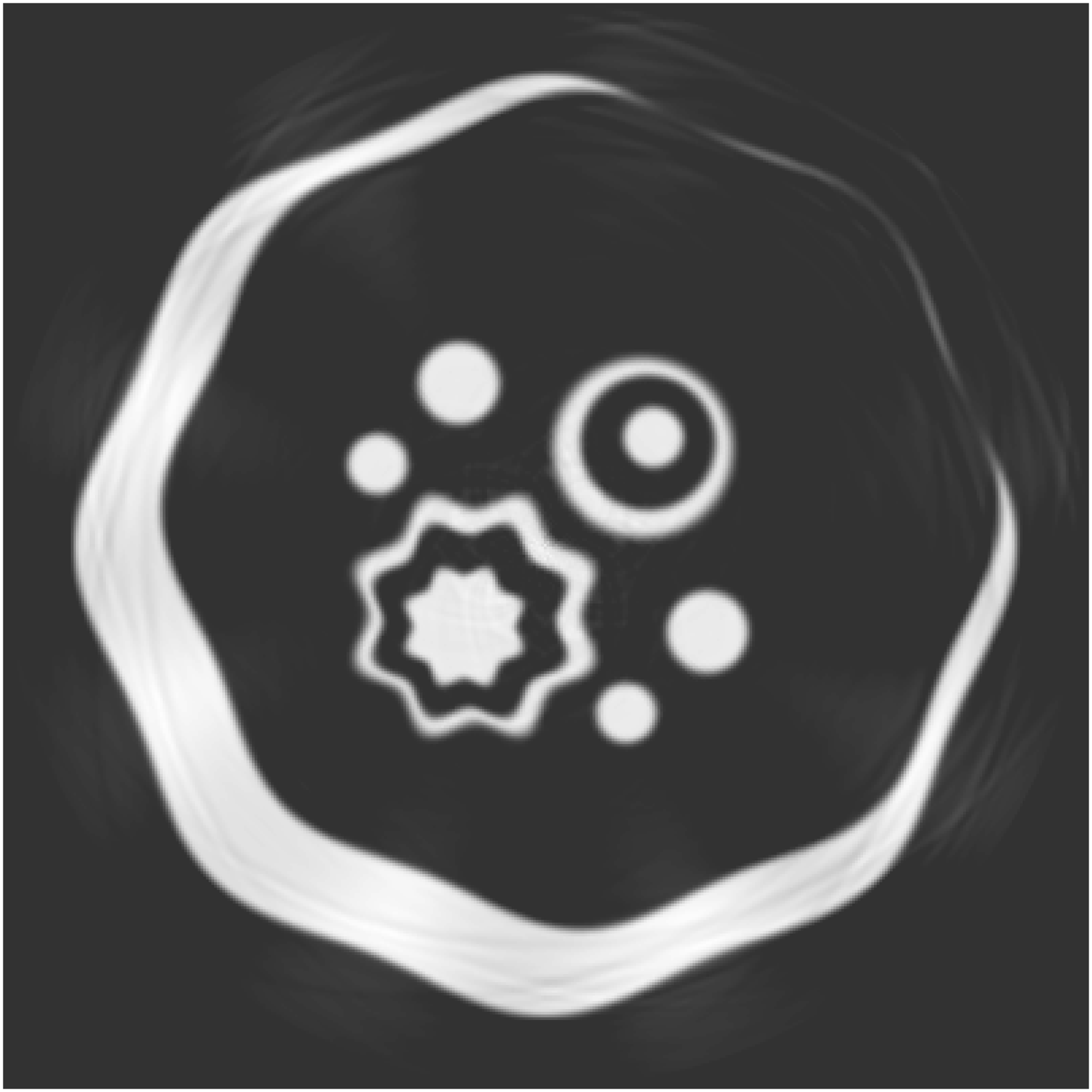}}
\end{center}
\caption{Simulation (a) phantom, interior problem (b) reconstruction (interior
problem) (c) phantom, interior/exterior problem (d) reconstruction in
interior/exterior problem. Dotted line shows locations of the transducers}%
\label{F:goodfronts}%
\end{figure}


\section{Numerical realization and simulations\label{S:numerics}}


We present below the results of numerical simulations illustrating the work of
the algorithms proposed in the previous sections. The first two series of
simulations test the performance of the CMT inversion technique developed in
Section~\ref{S:CMT}.

\subsection{Inversion of the CMT}

We first apply the CMT inversion algorithm to a standard interior problem.
This problem is well-posed and a good quantitatively correct reconstruction is
expected. As a phantom, we used a set of functions whose gray-scale image is
shown in Figure~2(a). These functions are mostly constant (equal 1) within
their support, however the transition from 1 to 0 is smooth. This smoothness
is especially important in the next section where the wave equation was solved
by finite difference methods whose accuracy would be severely compromised if
the test phantom were discontinuous. The simulated transducers are passing
through the circle of radius $0.5$ perpendicular to the plane of the Figure;
their locations are shown by the dotted line in Figure~2(a). Since the phantom
is completely surrounded by the transducers, this is an interior problem.
There were 256 simulated transducers in this experiment; for each of them 401
circular means where computed, with the radii ranging from 0 to 2. Figure~2(b)
shows reconstruction of the function from the circular means by the method
presented in Section~\ref{S:CMT}. As expected, the reconstruction is quite accurate.

In order to test the performance of our method in the exterior/interior
problem we added to the phantom an additional function with the intention to
crudely model aorta walls, as seen, for example, in \cite{Sethur-2007}. The
material interfaces modelled by this function are \textquotedblleft visible".
The reconstruction from the circular means is shown in Figure~2(d). One can
see that the interior part of the phantom is reconstructed as well as before.
In spite of the ill-posedness of the exterior problem and the inexact nature
of the regularized algorithm, the exterior part of the phantom is also
reconstructed quantitatively correct, although a careful reader will notice
some variations in the brightness of the \textquotedblleft aorta walls".

It is interesting to see what effect on the reconstruction will have the
presence of ``invisible" interfaces. To this end we added to the phantom
several circular inclusions (see Figure~3(a)). Figure~3(b) demonstrates the
reconstruction from the accurate circular means. In Figure~3(c) before the
reconstruction a simulated noise was added to the circular means; the
intensity of the noise was $5\%$ of the ``signal" in $L^{2}$ norm. For a fair
comparison images in Figure~3 are drawn using the same gray scale. Images in
Figures~3(b) and~3(c) show that, as expected, the ``invisible" material
interfaces are significantly blurred, and the details (circular inclusions)
are not reconstructed correctly. Nevertheless, the visible parts of the
corresponding boundaries are captured, and the artifacts are relatively well
localized, so that ``aorta walls" are still clearly seen.

\begin{figure}[t]
\begin{center}
\subfigure[]{\includegraphics[width=1.7in,height=1.7in]{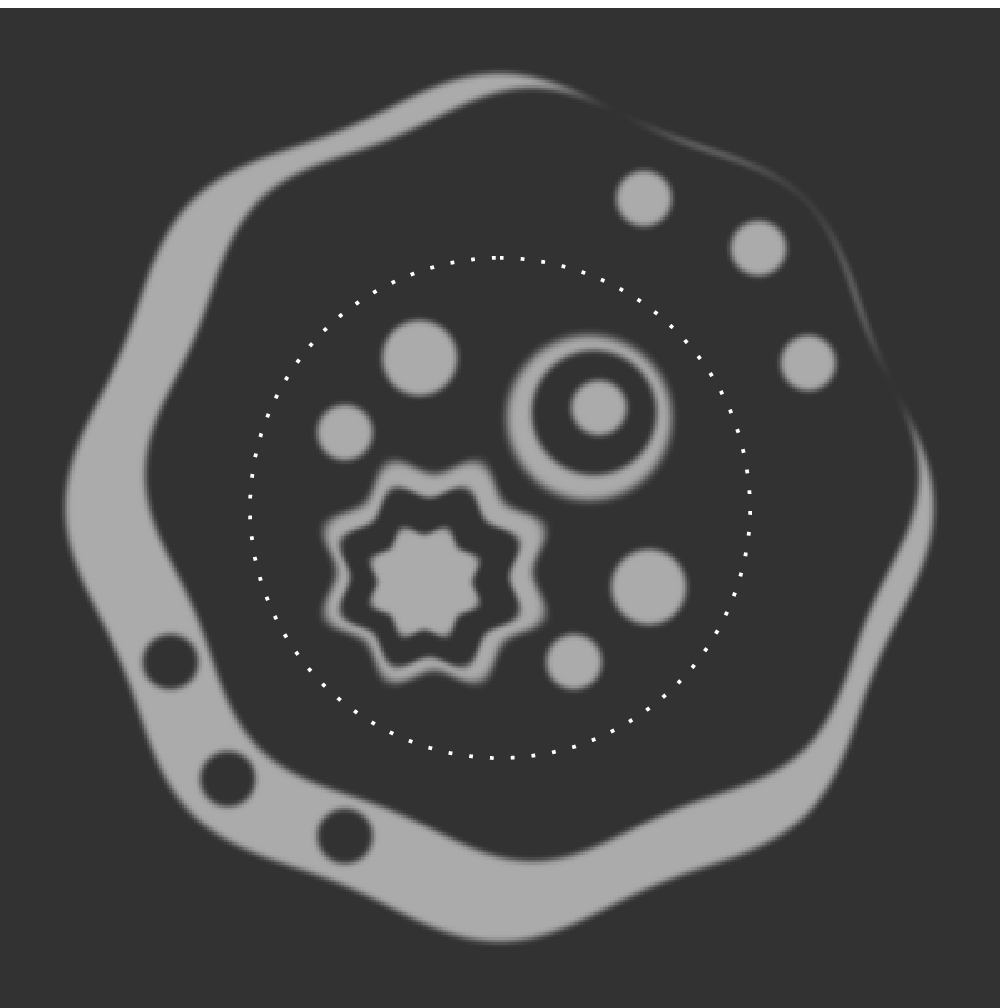}}
\subfigure[]{\includegraphics[width=1.7in,height=1.7in]{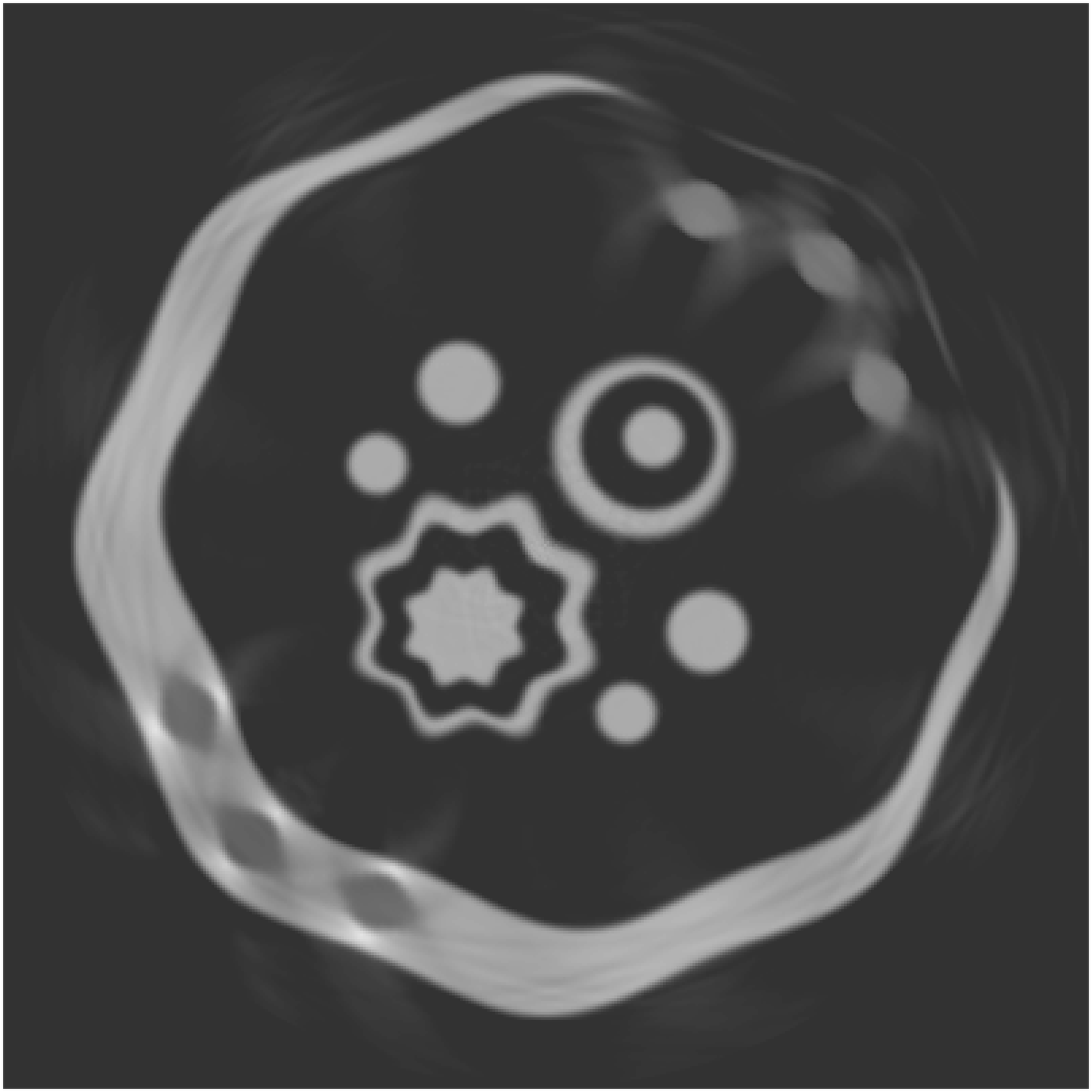}}
\subfigure[]{\includegraphics[width=1.7in,height=1.7in]{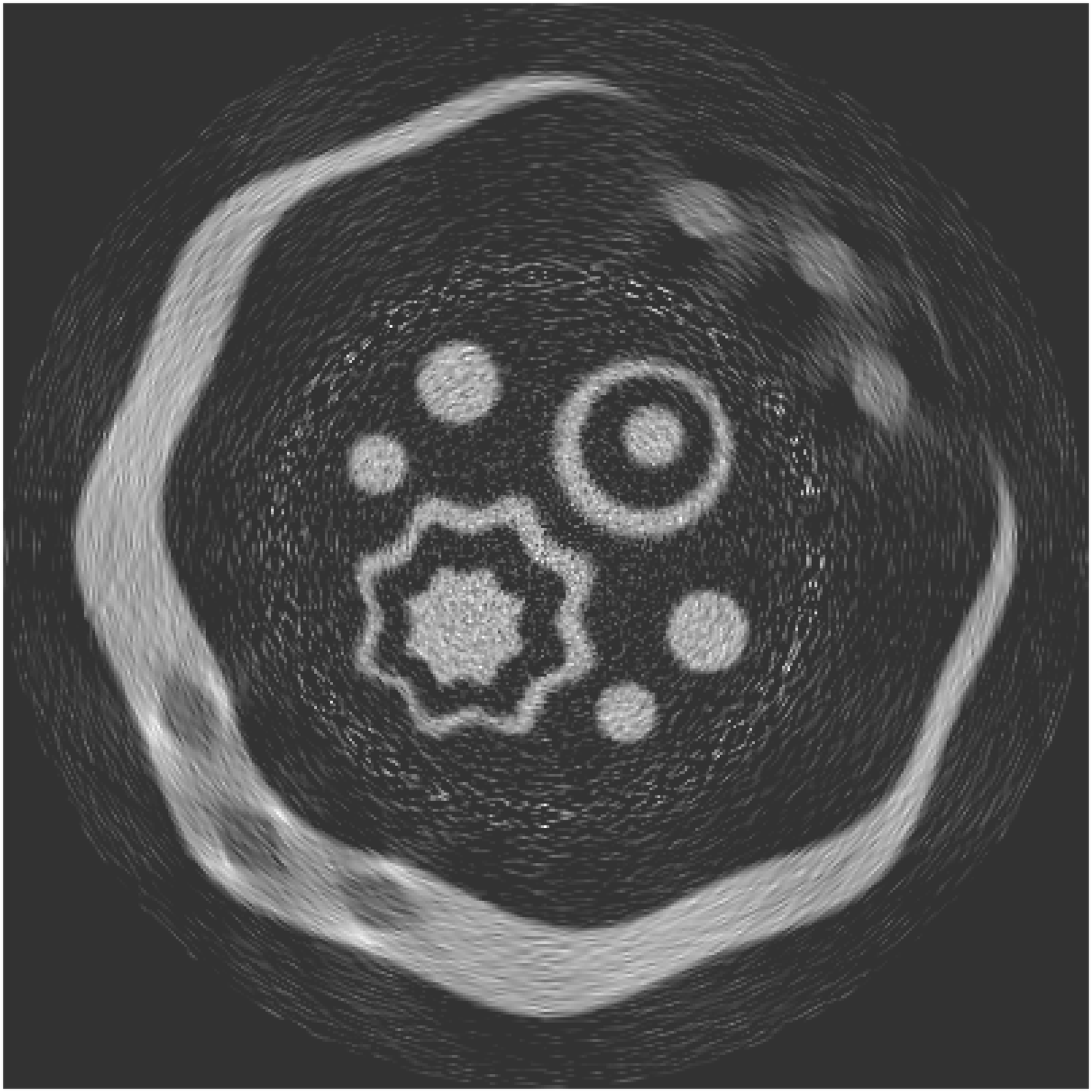}}
\end{center}
\caption{Simulation (a) phantom, dotted line shows locations of the
transducers (b) reconstruction from accurate data (c) reconstruction from data
with $5\%$ noise. }%
\label{F:badfronts}%
\end{figure}

\begin{figure}[t]
\begin{center}
\subfigure[]{\includegraphics[width=1.7in,height=1.7in]{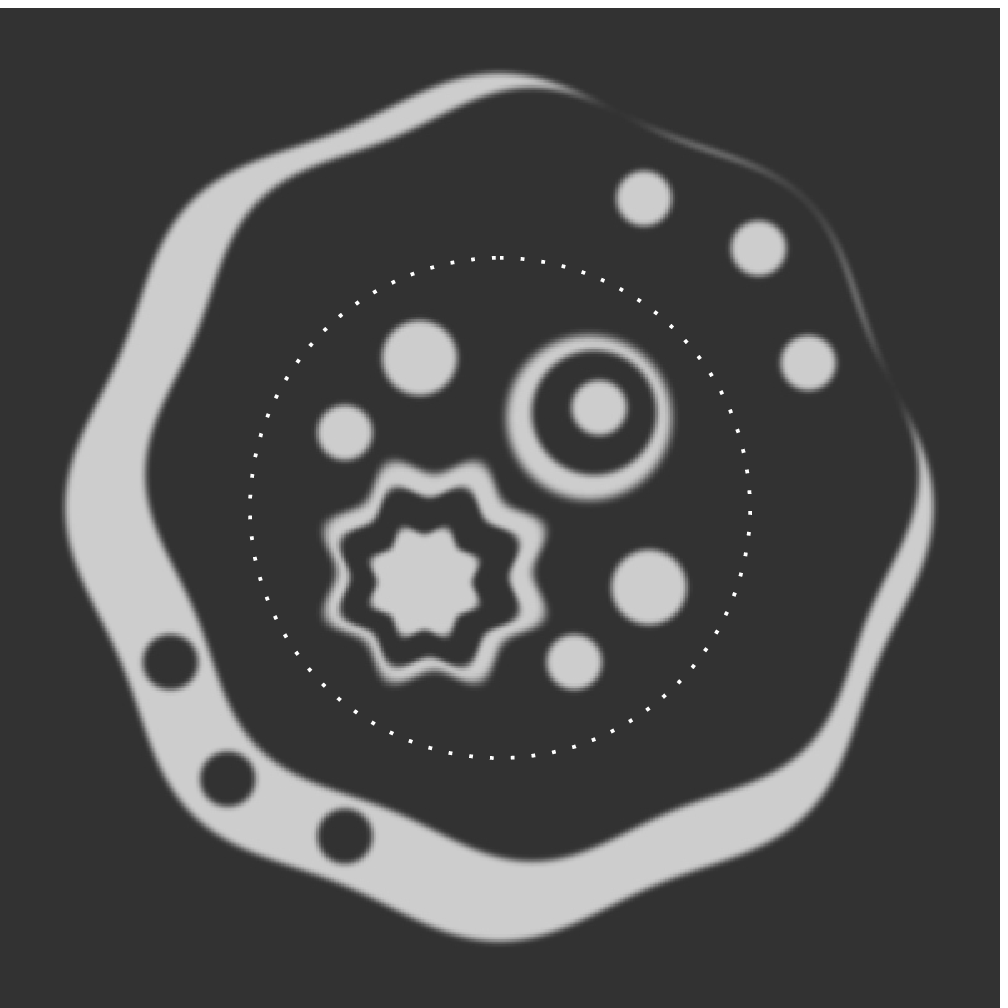}}
\subfigure[]{\includegraphics[width=1.7in,height=1.7in]{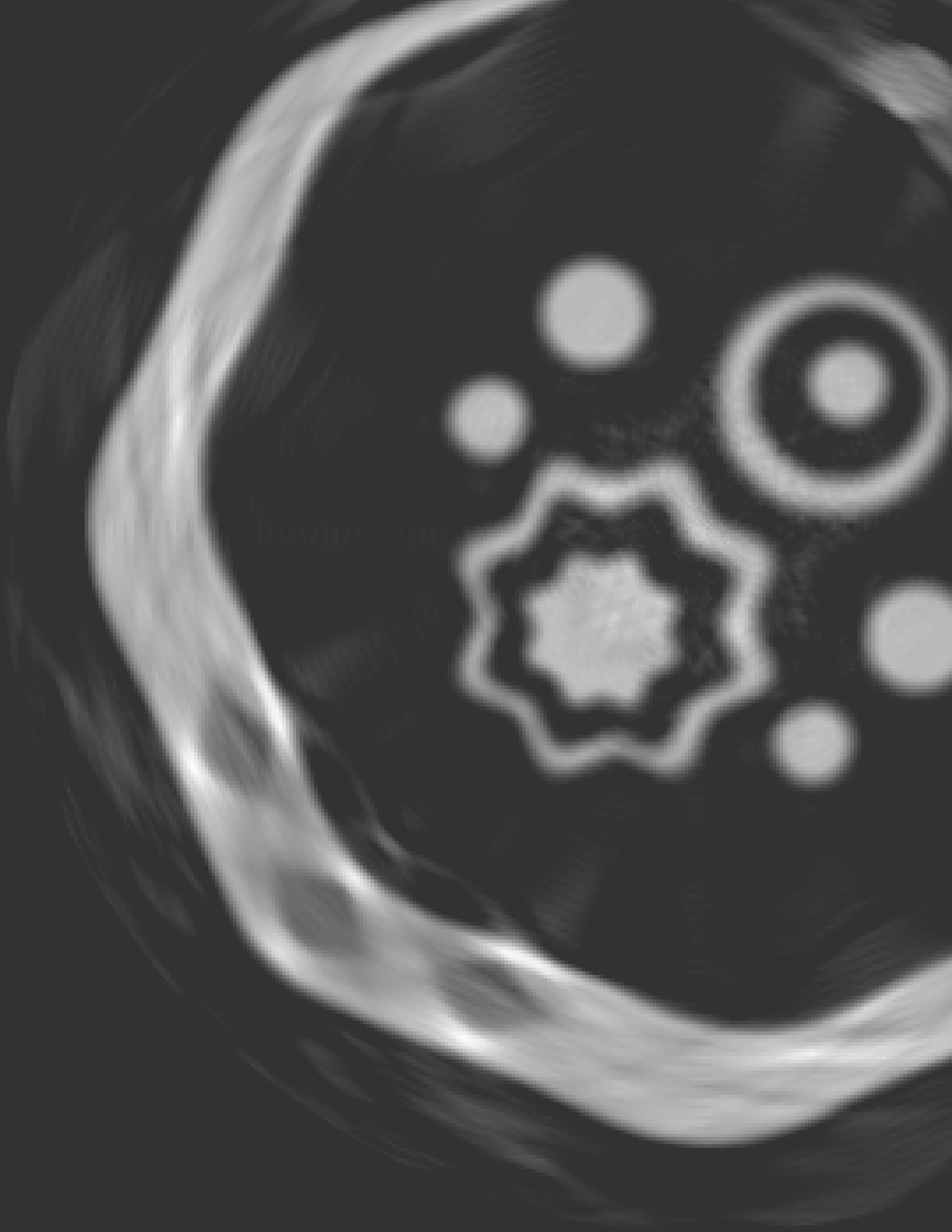}}
\subfigure[]{\includegraphics[width=1.7in,height=1.7in]{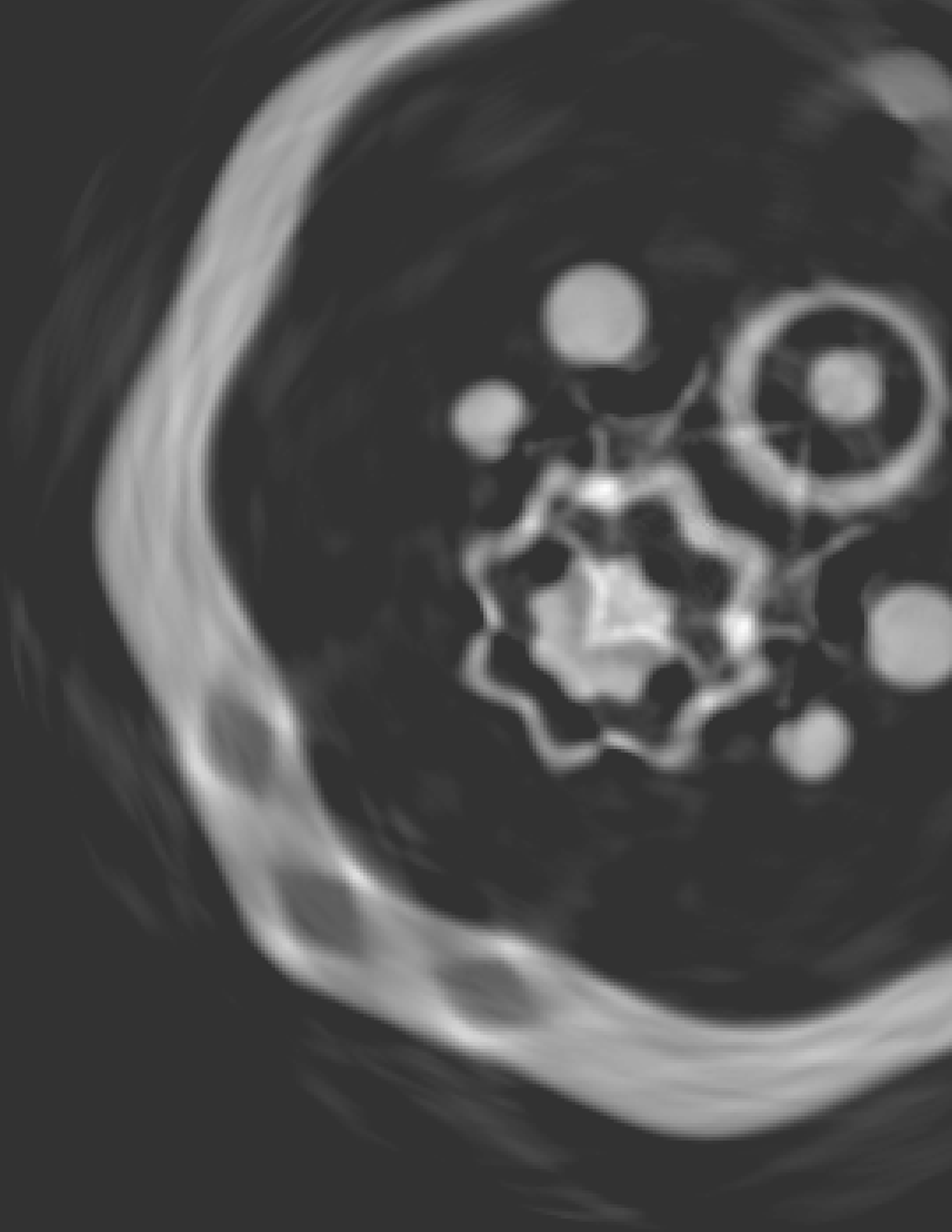}}
\end{center}
\caption{Reconstruction from the wave equation simulation (a) phantom, dotted
line shows locations of the transducers (b) reconstruction corresponding to
the maximum speed of 1.01 (c) reconstruction corresponding to the maximum
speed of 1.05. }%
\label{F:born}%
\end{figure}

\subsection{Simulation of the IVUS}

Our next simulation tests the feasibility of tomography reconstruction in
IVUS. The geometry of the simulation is similar to that of the previous
section. This time, however, for every transducer location using the
time-stepping finite-difference technique we solved two wave equations. One of
them corresponds to the uniform speed of sound; it models $u_{0}%
(t,\mathbf{x})$ from Section~\ref{S:Born}. Since the initial condition does
not depend on $x_{3},$ this is actually a 2D problem. The other equation we
solve is equation (\ref{E:full-wave}) with a variable speed of sound, again
not depending on $x_{3}.$The simulated transducers ``measured" the difference
$w(t,\mathbf{x})$ of these two solutions.

The variable speed of sound in the first of our simulations was equal to 1
plus 1\% perturbation modelled by the phantom from the previous section. The
perturbation is shown in Figure~4(a). The circular means of the perturbation
were reconstructed using the method from Section~\ref{S:Born}, and the
perturbation was then reconstructed from the circular means by the method of
Section~\ref{S:CMT}. The result is shown in Figure~4(b). Although additional
artifacts can be noticed in the image (as compared to Figure~3(b)), the
``aorta walls" are still clearly seen and the interior part of the image is
reconstructed well. The additional artifacts are caused by the imprecise
nature of the method (the Born approximation is used) and by the errors
introduced by finite differences when modelling the forward problem.

In order to better understand the effect of the Born approximation we repeated
the last simulation with new variable speed of sound equal 1 plus a 5\%
perturbation given by the same function. The result is shown in Figure~4(c)
(with the gray scale modified by the factor of 5). One can see additional
artifacts in the central part of the image; they are clearly caused by the
increased error in the Born approximation, since everything else remains the
same. Nevertheless the rest of the image is reconstructed quite well, and the
parts of the phantom with visible material interfaces are reconstructed with a
reasonable accuracy.

\section*{Concluding remarks}

We have considered in the present paper the inverse problems that arise in
IVPA and IVUS, and have shown that under several simplifying assumptions these
problems can be reduced to the solution of the exterior problem for the CMT.
The latter is, in general, severely ill-posed. However, the regularized
algorithm proposed in Section \ref{S:CMT} yields stable reconstruction of the
visible material interfaces while it blurs the invisible ones. If the
invisible interfaces are absent or almost absent in the image, the algorithm
reconstructs a quantitatively correct image. Such a situation (absence of
invisible interfaces) is not uncommon in practical application of IVPA and
IVUS, judging from the images we found in the literature.

In order to carry out our analysis we have made several simplifying
assumptions. In particular, the transducers were assumed infinitely long and
infinitely thin, the body of the catheter was assumed to be made of material
with the same speed of sound as that in blood. We also assumed that the speed
of sound in the soft tissues constituting and surrounding the vessels is close
to the speed of sound in blood. The latter assumption is standard in problems
of thermoacoustic tomography; it allowed us to use the constant speed
approximation in IVPA and the Born approximation in IVUS. The model of
acoustically transparent and infinitely long catheter is more crude; it was
used mostly to simplify the analysis and to develop practically useful
reconstruction algorithms. It is not clear whether it is feasible or practical
to manufacture the catheter whose speed of sound matches that of blood or
water. The effect of the finite length of the transducers also requires
further investigation.

However, our purpose was not to develop a perfect reconstruction algorithm,
but rather to demonstrate the feasibility of tomography-like reconstruction in
IVPA and IVUS. More sophisticated techniques, that would take into account the
details we neglected or over-simplified in the present, first approach to the
problem, will be the subject of the future work.

\section*{Acknowledgements}

The authors gratefully acknowledge support by the NSF through grants NSF/DMS
1109417 (the first author) and NSF/DMS 1211521 (the second author). We also
would like to thank E. T. Quinto for many helpful discussions.

\end{document}